\documentclass{elsarticle}

\usepackage{
amsmath,
amsthm,
amscd,
amssymb,
}
\usepackage{tikz}
\usetikzlibrary{positioning,arrows}

\newtheorem{lemma}{Lemma}
\newtheorem{definition}{Definition}
\newtheorem{corollary}{Corollary}
\newtheorem{proposition}{Proposition}
\newtheorem{example}{Example}
\newtheorem{theorem}{Theorem}

\newcommand{\Z}{\mathbb{Z}}

\newcommand{\N}{\mathbb{N}}
\newcommand{\B}{\mathcal{B}}
\newcommand{\ID}{\mbox{id}}
\newcommand{\INF}{{}^\infty}

\newcommand{\Ker}[1]{\mathrm{Ker}(#1)}

\newcommand{\Syn}{\mathrm{Syn}}

\newcommand{\lcm}{\mathrm{lcm}}

\newcommand{\Sub}{\mathrm{Sub}}
\newcommand{\cat}{\mathcal{C}}
\newcommand{\catt}{\mathcal{D}}
\newcommand{\perleq}{\swarrow}
\newcommand{\pergeq}{\searrow}
\newcommand{\per}{\mathrm{per}}

\begin{document}
\begin{frontmatter}

\title{Category Theory of Symbolic Dynamics\tnoteref{t1}}
\tnotetext[t1]{Research supported by the Academy of Finland Grant 131558}

\author[utu]{Ville~Salo\corref{vs}}
\ead{vosalo@utu.fi}

\author[utu]{Ilkka~T\"orm\"a\corref{it}}
\ead{iatorm@utu.fi}

\address[utu]{TUCS -- Turku Centre for Computer Science \\ Department of Mathematics and Statistics \\ 20014 University of Turku, Finland \\ +358 2 333 6012}

\cortext[vs]{Corresponding author}
\cortext[it]{Principal corresponding author}

\begin{abstract}
We study the central objects of symbolic dynamics, that is, subshifts and block maps, from the perspective of basic category theory, and present several natural categories with subshifts as objects and block maps as morphisms. Our main goals are to find universal objects in these symbolic categories, to classify their block maps based on their category theoretic properties, to prove category theoretic characterizations for notions arising from symbolic dynamics, and to establish as many natural properties (finite completeness, regularity etc.) as possible. Existing definitions in category theory suggest interesting new problems in symbolic dynamics. Our main technical contributions are the solution to the dual problem of the Extension Lemma and results on certain types of conserved quantities, suggested by the concept of a coequalizer.
\end{abstract}

\end{frontmatter}

\section{Introduction}

Like many branches of mathematics, symbolic dynamics is the study of a category: the category with subshifts as objects and block maps as morphisms. Of particular interest is the case where the objects are SFTs or sofic shifts, and we will mostly concentrate on such subcategories. Like in many branches of mathematics, the first question that comes to mind is still open: `When are two objects isomorphic?', and a lot of mathematics has been developed trying to answer this question \cite{Bo08}. Questions such as `When is an object a subobject of another?' (the embedding problem) and `When does an object map onto another?' (the factoring problem) have been solved at least in important special cases (see Factor Theorem and Embedding Theorem in \cite{LiMa95}, and the article \cite{Bo83}).

A nice feature of category theory is that it allows one to define notions such as `isomorphism', `embedding' and `factoring' without referring to anything except morphisms between objects. For example, two objects $X$ and $Y$ are isomorphic, in the sense of category theory, if there exist morphisms $f : X \to Y$ and $g : Y \to X$ such that $g \circ f = \ID_X$ and $f \circ g = \ID_Y$. Usually, this turns out to be the `correct' notion for isomorphism. For embedding (and factoring), the situation is more complicated: often the most natural definition of an embedding is that it is injective. The notion of injectivity (and surjectivity) is, however, impossible to define categorically, since a category does not know what its morphisms actually are: they need not even be functions! Because of this, multiple categorical variants of injectivity have been defined, including monicness, split monicness and regular monicness. These are generalizations of the various ways in which an injective map should behave in relation to other morphisms. In sufficiently nice categories (for example, the category of sets), these all correspond to injectivity, but in many categories, they state some other property, and often raise new natural questions.

Because categorical notions depend on \emph{all} morphisms of the category, we define several categories (thirteen, to be exact) whose objects are subshifts and morphisms are block maps between them. We begin our study in Section~\ref{sec:Morphisms} by considering the various categorical definitions of injectivity and surjectivity for our categories. We see that these categories are not nearly as `nice' as that of sets, in that usually these notions do not characterize injectivity or surjectivity in our categories. Interestingly, the characterization of split monicness comes from the well-known Extension Lemma. The dual concept of split epicness turns out interesting as well, and we find a characterization in the SFT case. Our discussion of morphisms is motivated as the study of how well standard notions of category theory describe the world of block maps, and in Section~\ref{sec:Other}, we ask a kind of converse question of whether category theory can describe standard notions of symbolic dynamics. Here, we mostly concentrate on the properties of objects. 

In Section~\ref{sec:Limits}, we move on to important category theoretical closure properties: the existence of (finite) limits and colimits of diagrams. The importance of these notions is that many category theoretical definitions are just limits of diagrams of certain types. The case of coequalizers turns out to be the most intricate, and we present an undecidability result related to it. In Section~\ref{sec:Properties}, we discuss the existence of images, disjoint unions and quotients in a categorical sense, that is, whether our categories are regular, coherent and exact, respectively.

The topics of cellular automata (particular kinds of block maps) and category theory have been previously explicitly discussed together at least in \cite{CaUu10}, but the approach there is very different. There, cellular automata are constructed using category theoretical tools, and properties that `come for free' from category theoretical generalities are investigated.

\section{Definitions and Notation}
\label{sec:Defs}

In this section, we establish the basic terminology and notation used in this article. As we are working in the intersection of two quite distinct fields, symbolic dynamics and category theory, we try to be as complete as possible.

\subsection{Symbolic Dynamics}

For a finite set $S$ (an \emph{alphabet}) with the discrete topology, we denote by $S^\Z$ the space of two-way infinite \emph{configurations} (or \emph{points}) over $S$ with the product topology, and call it the \emph{full shift on $S$}. The \emph{shift action} $\sigma : S^\Z \to S^\Z$ is defined by $\sigma(x)_i = x_{i+1}$ for all $i \in \Z$. A fixed point of $\sigma$ is called a \emph{uniform point}. A closed subset $X$ of a full shift with $\sigma(X) = X$ is called a \emph{subshift}. We say that a word $w \in S^*$ occurs in $x \in S^\Z$ if $x_{[i,i+|w|-1]} = w$ for some $i \in \Z$, and also write $w \sqsubset x$. An alternative characterization of subshifts is by means of \emph{forbidden words}: $X \subset S^Z$ is a subshift if and only if there exists a set of words $F \subset S^*$ such that $X = \{x \;|\; \forall w \in F: w \not\sqsubset x\}$. If the set of forbidden words can be taken to be finite, $X$ is called a \emph{subshift of finite type}, or SFT for short, and if the set is a regular language, $X$ is called \emph{sofic}. If $X$ is an SFT and $Y$ is any subshift, then $X \cap Y$ is a \emph{subSFT} of $Y$. If the SFT (or subSFT of $Y$) $X$ can be defined by forbidden words of length at most $m$ (in addition to the forbidden words of $Y$), we say $m$ is a \emph{window size} of $X$ (relative to $Y$). We denote by $\sigma_X$ the restriction of $\sigma$ to $X$.

The words of length $n$ appearing in configurations of $X$ are denoted $\B_n(X)$, and we denote the \emph{language} of $X$ by $\B(X) = \bigcup_{n \in \N} \B_n(X)$. Since a subshift is defined by its language \cite{LiMa95}, we may also denote $X = \B^{-1}(L)$, if $L \subset S^*$ is an extendable language (for every $v \in L$ there exist $u, w \in S^+$ with $uvw \in L$) such that $\B(X)$ is the language of subwords of words in $L$. Usually, when using this notation, we write a regular expression in place of $L$. An SFT $X \subset S^\Z$ can also be defined by giving a set of \emph{allowed words} $A \subset S^n$ for some $n \in \N$ such that $X = \{x \in S^\Z \;|\; \forall i: x_{[i,i+n-1]} \in A\}$.

\begin{example}
The sofic shift $\B^{-1}(0^*10^*)$ consists of exactly those configurations of $\{0,1\}^\Z$ that contain at most one $1$. The subshift $\B^{-1}((0^*10)^*)$ is an SFT, and can be defined by the single forbidden word $11$. For a fixed $p \in \N$, the SFT $\B^{-1}((0^{p-1} 1)^*)$ contains exactly $p$ points with spatial period $p$, and we use it as a `canonical' $p$-periodic subshift. As a dynamical system, it is isomorphic to $(\Z_p, n \mapsto n + 1 \bmod p)$, the set of integers modulo $p$ with incrementation.
\end{example}

For two subshifts $X \subset S^\Z$ and $Y \subset R^\Z$, define $X \times Y \subset (S \times R)^\Z$ as the coordinatewise product
\[ \{ z \in (S \times R)^\Z \;|\; \ldots p_1(z_{-1}) p_1(z_0) p_1(z_1) \ldots \in X \wedge \ldots p_2(z_{-1}) p_2(z_0) p_2(z_1) \ldots \in Y \}, \]
where $p_1$ and $p_2$ are the appropriate projections from $S \times R$. Define also $X \mathop{\dot{\cup}} Y$ as their symbol-disjoint union, where we replace $S$ and $R$ with disjoint sets if necessary.

The \emph{syntactic monoid} $\Syn(X)$ of a subshift $X \subset S^\Z$ is defined as $S^* / {\sim_X}$, where $u \sim_X v$ denotes that $w u w' \in \B(X)$ if and only if $w v w' \in \B(X)$ for all $w, w' \in S^*$. We also denote $(v)_X = v / {\sim_X}$. It is known that sofic shifts are exactly those subshifts whose syntactic monoid is finite.

A \emph{block map} is a continuous function $f : X \to Y$ from a subshift to another with $f \circ \sigma_X = \sigma_Y \circ f$. Block maps are defined by \emph{local functions} $F : \B_{2r+1}(X) \to \B_1(Y)$ by $f(x)_i = F(x_{[i-r,i+r]})$, where $r \geq 0$ is called a \emph{radius} of $f$. If $f$ is surjective, $Y$ is a \emph{factor} of $X$, and if it is bijective, $X$ and $Y$ are \emph{conjugate}. The block map itself is called an \emph{embedding}, a \emph{factor map} or a \emph{conjugacy}, if it is injective, surjective or bijective, respectively. We say $f$ is \emph{preinjective} if $f(x) \neq f(y)$ whenever $x \neq y \in X$ are \emph{asymptotic}, that is, they differ in finitely many coordinates. If $X = Y$, then $f$ is called a \emph{cellular automaton} on $X$.

A nonempty subshift $X$ is \emph{minimal} if it does not properly contain another nonempty subshift. A subshift $X$ is \emph{transitive} (called \emph{irreducible} in \cite{LiMa95}) if for all $u,v \in \B(X)$ there exists $w \in \B(X)$ such that $uwv \in \B(X)$, and \emph{nonwandering} if for all $u \in \B(X)$ there exists $w \in \B(X)$ such that $uwu \in \B(X)$. It is \emph{mixing} if for all $u,v \in \B(X)$ there exists $N \in \N$ such that for all $n \geq N$ there exists $w \in \B_n(X)$ with $uwv \in \B(X)$.

It is known that for a mixing sofic shift, $N$ can be chosen independently of $u$ and $v$, and is called its \emph{mixing distance}. For a transitive SFT $X$, something similar is true:\footnote{We do not need these kinds of tools for general transitive sofic shifts.} We define $\per(X) = \gcd\{|w| \;|\; \INF w \INF \in X\}$, called the \emph{period} of $X$. It is well-known (although usually stated rather differently) that for all $p \in \N$, there exists a block map $\phi : X \to \B^{-1}((0^{(p-1)}1)^*)$ if and only if $p$ divides $\per(X)$ \cite[Section~4.5]{LiMa95}. Let $p = \per(X)$, and call $\phi$ the \emph{phase map} of $X$. There also exists a number $m \in \N$, a multiple of $p$, such that for all $u, v \in \B(X)$ there exists $k \in [0, p-1]$ and $w \in \B_{m+k}(X)$ such that $u w v \in \B(X)$. We call such an $m$ a \emph{transition distance} of $X$. If $\phi$ is the associated phase map, writing $X_0 = \phi^{-1}(\INF (0^{p-1}1) \INF)$ and $X_i = \sigma(X_{i-1})$ for $0 < i < p$, we have $X = \bigcup_{i = 0}^{p-1} X_i$, where the union is disjoint. If $r \in \N$ is the radius of $\phi$, then given a word $u \in \B(X)$ of length at least $2r+p$, there exists a unique number $i(u) \in [0, p-1]$, called the \emph{phase} of $u$, such that $x_{[0,|u|-1]} = u$ for some $x \in X_{i(u)}$. For two words $u, v \in \B(X)$ both of at least this length, there in fact exists a unique $k \in [0, p-1]$ such that $u w v \in \B(X)$ for some $w \in \B_{m+k}(X)$, namely $k = i(v) - i(u) - |u| \bmod p$.

A cellular automaton $f : X \to X$ is \emph{(topologically) mixing}, if for all $u, v \in \B(X)$, and for all large enough $n \in \N$ depending on the words, there exists $x \in X$ with $x_{[0, |u|-1]} = u$ and $f^n(x)_{[0, |v|-1]} = v$. It is \emph{chain transitive} if for all $n \in \N$ and $u, v \in \B_n(X)$, there exists a chain $x^1, \ldots, x^k \in X$ such that $x^1_{[0,n-1]} = u$, $x^k_{[0,n-1]} = v$ and for all $i$, $f(x^i)_{[0,n-1]} = x^{i+1}_{[0,n-1]}$. A CA $f : X \to X$ is \emph{sensitive (to initial conditions)} if there exists $k \in \N$ such that for all $uv \in \B(X)$ there exist $x, y \in X$ with $x_{[-|u|,|v|-1]} = y_{[-|u|,|v|-1]} = uv$ and $n \in \N$ such that $f^n(x)_{[0,k-1]} \neq f^n(y)_{[0,k-1]}$.The \emph{limit set} of the cellular automaton $f$ is the subshift $\bigcap_{n \in \N} f^n(X)$, and $f$ is \emph{stable} if the limit set equals $f^n(X)$ for some $n \in \N$.

For an arbitrary function $f : A \to A$, define $\Ker{f} = \{ (a, b) \;|\; f(a) = f(b) \} \subset A \times A$, and call it the \emph{kernel set} of $f$. For a set $A$, denote the \emph{diagonal} of $A$ by $\Delta_A = \{ (a, a) \;|\; a \in A \} \subset A \times A$.

\subsection{Category Theory}

We now recall some definitions of category theory. The standard reference for the subject is \cite{Ma71}, while \cite{Jo02} and \cite{Jo02a} were consulted for some of the less standard notions.

In this article, a \emph{category} $\cat$ consists of a class of \emph{objects}, and a class of \emph{morphisms}. Every morphism $f$ has a \emph{source object} $X$ and a \emph{target object} $Y$, and we write $f : X \to Y$. The set of morphisms from $X$ to $Y$ is denoted $\mathrm{Hom}_\cat(X,Y)$. Two morphisms $f : X \to Y$ and $g : Y \to Z$ can be \emph{composed} to obtain a morphism $g \circ f : X \to Z$. We sometimes suppress the operator ${\circ}$ and write $g f$ for $g \circ f$. The two axioms of a general category require that composition is associative (so $h (g f) = (h g) f$ when the compositions are defined), and that every object $X$ has a \emph{identity morphism} $\ID_X : X \to X$ with $f \circ \ID_X = f$ and $\ID_X \circ g = g$ for all $f : X \to Y$ and $g : Z \to X$. A category is \emph{concrete} if its morphisms are actually functions between the objects.

A morphism $f : X \to Y$ in a category $\cat$ is
\begin{itemize}
\item \emph{epic} or an \emph{epimorphism} if $g \circ f \neq h \circ f$ for all $g \neq h : Y \to Z$,
\item \emph{monic} or a \emph{monomorphism} if $f \circ g \neq f \circ g$ for all $g \neq h : Z \to X$,
\item \emph{split epic} if there exists $g : Y \to X$ with $f \circ g = \ID_Y$ (a \emph{section} of $f$),
\item \emph{split monic} if there exists $g : Y \to X$ with $g \circ f = \ID_X$ (a \emph{retract} of $f$),
\item an \emph{isomorphism} if there exists $g : Y \to X$ with $f \circ g = \ID_Y$ and $g \circ f = \ID_X$ (an \emph{inverse} of $f$).
\end{itemize}
Epimorphisms and monomorphisms are generalizations of surjective and injective functions, respectively.

A \emph{functor} $F$ from one category $\cat$ to another category $\catt$ associates to each object $X$ of $\cat$ an object $F(X)$ of $\catt$, and to each morphism $f : X \to Y$ of $\cat$ a morphism $F(f) : F(X) \to F(Y)$ of $\catt$, in such a way that $F(\ID_X) = \ID_{F(X)}$ and $F(f \circ g) = F(f) \circ F(g)$ always hold. It is \emph{full} (\emph{faithful}) if the induced map from $\mathrm{Hom}_\cat(X,Y)$ to $\mathrm{Hom}_\cat(F(X),F(Y))$ is surjective (injective, respectively) for all objects $X$ and $Y$ of $\cat$. It is \emph{essentially surjective} if for every object $X$ of $\catt$, there exists an object $Y$ of $\cat$ such that $F(Y)$ is isomorphic to $X$. We say $\cat$ and $\catt$ are \emph{equivalent} is there exists a full, faithful and essentially surjective functor from $\cat$ to $\catt$. It is well known that this notion is actually symmetric.

Let $f : Y \to X$ and $g : Z \to X$ be monomorphisms in a category $\cat$. We denote $f \leq g$ if there exists a morphism $h : Y \to Z$ such that $f = g \circ h$. Such an $h$ must be monic since $f$ is, and it is unique since $g$ is monic. If $h$ is an isomorphism, we say $f$ and $g$ are \emph{isomorphic}. A \emph{subobject} of $X$ is an isomorphism class of monomorphisms into $X$, and the class of all subobjects of $X$, denoted $\Sub(X)$, is partially ordered by $\leq$. We can also view $\Sub(X)$ as a category whose morphisms are the morphisms $h$ of $\cat$ as above.

A \emph{diagram} in a category $\cat$ is formally a functor from another category $\mathcal{I}$ to $\cat$. A \emph{cone} of a diagram $D : \mathcal{I} \to \cat$ is an object $C$ of $\cat$, together with morphisms $\phi_X : C \to D(X)$ for all objects $X$ of $\mathcal{I}$, such that for all morphisms $f : X \to Y$ in $\mathcal{I}$ we have $\phi_Y = D(f) \circ \phi_X$. The cone is a \emph{limit} of $D$, if for all other cones $C', (\phi'_X)_X$, there is a unique morphism $u : C' \to C$ with $\phi'_X = \phi_X \circ u$ for all objects $X$ of $\mathcal{I}$. The notions of \emph{co-cone} and \emph{colimit} are defined dually, that is, with the morphisms reversed. Limits and colimits of diagrams are unique up to a unique isomorphism.

For a category $\cat$, we define the following limits and colimits:
\begin{itemize}
\item a limit of the empty diagram is a \emph{terminal object}, and its colimit is an \emph{initial object},
\item a limit of the discrete diagram $X \enspace Y$ is a \emph{product} of $X$ and $Y$, and its colimit is their \emph{coproduct},
\item a limit of the diagram $X \stackrel{f}{\to} Z \stackrel{g}{\leftarrow} Y$ is a \emph{pullback} of $f$ and $g$, and the limit object is denoted by $X \times_Z Y$. The resulting morphism from $X \times_Z Y$ to $Y$ is the pullback of $f$ along $g$,
\item a limit of the diagram $X \rightrightarrows Y$ is an \emph{equalizer} of the two morphisms, and its colimit is their \emph{coequalizer}, and
\item a limit of the infinite diagram $X_0 \leftarrow X_1 \leftarrow X_2 \leftarrow \cdots$ is an \emph{inverse limit}.
\end{itemize}
In particular, an object $T$ is terminal (initial) if and only if for all objects $X$, there exists exactly one morphism from $X$ to $T$ (from $T$ to $X$, respectively). An object $Z$ which is both initial and terminal is called a \emph{zero object}, and for two objects $X$ and $Y$, the unique morphism from $X$ to $Y$ which factors through $Z$ is called a \emph{zero morphism}, and denoted $0_{X Y}$. The \emph{kernel pair} of a morphism $f$ is the pullback of $f$ with itself. The category $\cat$ is \emph{finitely (co-)complete}, if it has all (co)limits of finite diagrams (equivalently, a terminal (initial) object, all binary (co)products and (co)equalizers \cite{Ma71}).

Let $\cat$ be a category. A morphism $f : X \to Y$ in $\cat$ is \emph{regular epic (monic)} if it is the coequalizer (equalizer, respectively) of some pair of parallel morphisms. The category $\cat$ is \emph{regular} if
\begin{itemize}
\item $\cat$ is finitely complete,
\item the coequalizer of every kernel pair exists, and
\item the pullback of a regular epimorphism along any morphism is a regular epimorphism.
\end{itemize}

Let $f : X \to Y$ be any morphism in a finitely complete category $\cat$. Then, $f$ induces a functor $f_* : \Sub(Y) \to \Sub(X)$, called the \emph{base change functor}, by sending each subobject of $Y$ to its pullback along $f$. We say $\cat$ is \emph{coherent} if
\begin{itemize}
\item $\cat$ is regular,
\item every subobject poset $\Sub(X)$ has binary unions (least upper bounds), and
\item the binary unions are stable under base changes.
\end{itemize}

A \emph{congruence}, or \emph{internal equivalence relation}, on an object $X$ of a finitely complete category $\cat$ is an object $R$ together with four morphisms:
\begin{itemize}
\item a monomorphism $e : R \to X \times X$ (the embedding of $R$ into $X \times X$),
\item a morphism $r : X \to R$ that is a section to both $p_1 \circ e$ and $p_2 \circ e$ (the reflectivity morphism),
\item a morphism $s : R \to R$ such that $p_1 \circ e \circ s = p_2$ and $p_2 \circ e \circ s = p_1$ (the symmetry morphism), and
\item a morphism $t : R \times_X R \to R$, where $R \times_X R$ with the projections $q_1$ and $q_2$ is the pullback of the pair $(p_2 \circ e, p_1 \circ e)$, such that $p_1 \circ e \circ q_1 = p_1 \circ e \circ t$ and $p_2 \circ e \circ q_2 = p_2 \circ e \circ t$ (the transitivity morphism).
\end{itemize}
In particular, every kernel pair $X \times_Y X$ of a morphism $f : X \to Y$ is a congruence. A congruence is \emph{effective} if it is isomorphic to a kernel pair.
A category $\cat$ is \emph{exact} if it is regular and every congruence is effective.

A category $\cat$ is \emph{extensive} if it has all finite coproducts and pullbacks of coproduct injections, and in every commutative diagram
\begin{center}
\begin{tikzpicture}[node distance=1.25cm, auto]
	\node (a) {$A$};
	\node (ab) [right of=a] {$A \amalg B$};
	\node (b) [right of=ab] {$B$};
	\node (x) [above=.5cm of a]{$X$};
	\node (z) [right of=x] {$Z$};
	\node (y) [right of=z] {$Y$};
	\draw[->] (a) to node {} (ab);
	\draw[->] (b) to node {} (ab);
	\draw[->] (x) to node {} (z);
	\draw[->] (y) to node {} (z);
	\draw[->] (x) to node {} (a);
	\draw[->] (z) to node {} (ab);
	\draw[->] (y) to node {} (b);
\end{tikzpicture}
\end{center}
where $A \amalg B$ denotes the coproduct of $A$ and $B$, the two squares are pullback diagrams if and only if the top row is also a coproduct diagram. A reference for extensive categories is \cite{CaLaWa93}.

\section{Preliminary Results}
\label{sec:Prelims}

In this section, we establish some general lemmas and notions that will be of use later on.

\subsection{Kernel Sets and Factoring}

We start with a general lemma about compact Hausdorff spaces, inspired by the well-known corresponding result in the concrete category \textbf{Set}.

\begin{lemma}
\label{lem:ContConnector}
Let $f : X \to Y$ and $g : X \to Z$ be continuous functions with $\Ker{f} \subset \Ker{g}$, where $X$ is compact and $Y$ is Hausdorff. Then there exists a unique continuous function $u : f(X) \to g(X)$ such that $u \circ f = g$.
\end{lemma}

\begin{proof}
Since $X$ is compact, also the quotient spaces $X/{\Ker{f}}$ and $X/{\Ker{g}}$ are compact, and we have a continuous surjection $h : X/{\Ker{f}} \to X/{\Ker{g}}$ defined by $h(x/{\Ker{f}}) = x/{\Ker{g}}$. Now the function $f_* : X/{\Ker{f}} \to f(X)$ defined by $f_*(x/{\Ker{f}}) = f(x)$ is a continuous bijection from a compact space to a Hausdorff space, and thus has a continuous inverse, which we denote by $f_*^{-1}$. Then, $u = g \circ h \circ f_*^{-1}$ is a continuous function from $f(X)$ to $g(X)$ that satisfies $u \circ f = g$. Its uniqueness is clear from the proof.
\end{proof}

For us, the importance of Lemma~\ref{lem:ContConnector} is of course that it directly applies to the world of block maps.

\begin{corollary}
\label{cor:ContConnector}
Let $X, Y$ and $Z$ be subshifts and $f : X \to Y$ and $g : X \to Z$ block maps such that $\Ker{f} \subset \Ker{g}$. Then there exists a unique block map $u : f(X) \to g(X)$ such that $u \circ f = g$.
\end{corollary}

\begin{proof}
The existence and uniqueness of a continuous $u$ is given by Lemma~\ref{lem:ContConnector}. We also have
\[ (\sigma_Z \circ u \circ \sigma_Y^{-1}) \circ f = \sigma_Z \circ u \circ f \circ \sigma_X^{-1} = \sigma_Z \circ g \circ \sigma_X^{-1} = g, \]
so by uniqueness $\sigma_Z \circ u \circ \sigma_Y^{-1} = u$, and $u$ is a block map.
\end{proof}

\subsection{Kernel Sets and Equivalence Relations}

\begin{definition}
A \emph{subshift relation} between subshifts $X$ and $Y$ is a subshift $R$ of $X \times Y$. We say $R$ is a \emph{subshift equivalence relation} if $X = Y$ and it is also an equivalence relation, that is, we have $\Delta_X \subset R$, $(y,x) \in R$ for all $(x,y) \in R$, and $(x,z) \in R$ whenever $(x,y), (y,z) \in R$ for some $y \in X$. A subshift equivalence relation $R \subset X^2$ is said to be \emph{local} if there exists $n \in \N$ and an equivalence relation $E \subset \B_n(X)^2$ such that $R$ is defined by the forbidden words $(S^n)^2 \setminus E$ (and those of $X^2$).
\end{definition}

Clearly, a local equivalence relation $R \subset X^2$ is a subSFT of $X^2$, and a subSFT equivalence relation is local if and only if it is the kernel set of some block map $f : X \to Y$. Note that if $R$ is defined by an equivalence relation in $\B_n(X)$, it can also be defined by an equivalence relation in $\B_m(X)$ for any $m \geq n$, implying that local subSFT relations are closed under finite intersections.

\begin{example}
Let $S = \{0,1,2,3,4,5\}$, and consider the SFT $X \subset S^\Z$ defined by the allowed words $\{00,01,02,03,14,24,25,35,40,50\}$ of length $2$. Define the subSFT relation $R \subset X^2$ by the allowed single-letter words $(1,2)$, $(2,1)$, $(2,3)$, $(3,2)$ and $(s,s)$ for all $s \in S$. It can be checked that $R$ is an equivalence relation, and we can choose $E$ to be the set of allowed words of $R$, plus $(1,3)$ and $(3,1)$, and $R$ is thus local. Note that the extra words of $E$ do not actually occur in any configuration of $R$.
\end{example}

We also have an example of a subSFT equivalence relation which is not local.

\begin{example}
\label{ex:NotLocal}
Consider the allowed words
\[
\left(\begin{array}{ccc}a&b&c\\a&b&c\end{array}\right),
\left(\begin{array}{ccc}a&b&c\\a&b&1-c\end{array}\right),
\left(\begin{array}{ccc}a&b&1-b\\a&1-b&b\end{array}\right), \]
\[
\left(\begin{array}{ccc}a&1-a&1-a\\1-a&a&a\end{array}\right),
\left(\begin{array}{ccc}a&a&a\\1-a&1-a&1-a\end{array}\right)
\]
where $a, b, c \in \{0, 1\}$, forming the SFT relation $R \subset (\{0,1\}^\Z)^2$. That is, $(x, y) \in R$ if and only if $x = y$ or for some $i \in \Z$, we have $x_{(-\infty, i)} = y_{(-\infty, i)}$, $x_{[i,\infty)} = ab^\infty$, $y_{[i,\infty)} = ba^\infty$ and $a \neq b$. On one-sided sequences, this is the relation of being the binary representation of the same number. If is easy to check that this is a transitive relation, as its orbits are of sizes $1$ or $2$.

Now, suppose $R = \Ker{f}$ for some block map $f$, where $f$ has radius $r$. Suppose $u \in \{0, 1\}^r$ is the binary representation of $k$ and $v \in \{0, 1\}^r$ that of $k + 1$. Then $u1^\infty \sim v0^\infty$, so that $f(u1^\infty) = f(v0^\infty)$. In particular, $f(u) = f(v)$. Since $u$ and $v$ were representations of any successive numbers between $0$ and $2^r-1$, $f$ must be a trivial map, a contradiction since $R$ is not the full relation. Thus, $R$ is not local.
\end{example}

We can also express some properties of block maps using their kernel sets. As an example, we characterize preinjectivity in the SFT case. For this and future use, we give the following definition, which comes from the general theory of topological dynamics.

\begin{definition}
A \emph{transitive component} of a subshift $X$ is a transitive subshift $Y \subset X$ which is maximal with respect to inclusion among the set of transitive subshifts of $X$. A transitive component which is also mixing is simply called a \emph{mixing component}.
\end{definition}

The transitive components of SFTs are exactly their irreducible components as defined in \cite[Section 4.4]{LiMa95}, and in particular, they form a finite set of mutually disjoint SFTs. Sofic shifts also have a finite number of transitive components, but they may not be mutually disjoint. A notion of irreducible components of sofic shifts was also defined in \cite{Th04}, but in a completely different way. Also, we note that for sofic shifts, mixing components and maximal mixing subshifts are different notions.

\begin{example}
Define a sofic shift by $X = \B^{-1}((00 + 01)^* + (00 + 02)^*)$. Then $X$ has two transitive components, $\B^{-1}((00 + 01)^*)$ and $\B^{-1}((00 + 02)^*$, which have the nonempty intersection $Y = \{\INF 0 \INF\}$. Furthermore, $Y$ is a maximal mixing subshift of $X$, but it is not a mixing component, since it is not a maximal transitive subshift.
\end{example}

\begin{lemma}
\label{lem:KernelComponent}
A block map $f : X \to Y$, where $X$ is a transitive SFT, is preinjective if and only if $\Delta_X$ is a transitive component of $\Ker{f}$.
\end{lemma}

\begin{proof}
The subshift $\Delta_X$ is transitive, so it is contained in some transitive component $Z$ of $\Ker{f}$. We show that $Z = \Delta_X$ if and only if $f$ is preinjective. First, suppose that $Z \neq \Delta_X$, so that there exists a word $u \in \B(Z) \setminus \B(\Delta_X)$. Let $x \in \Delta_X$, and let $m \in \N$ be a window size for $X$ and $Z$. Since $Z$ is a transitive SFT, there exist $v, w \in \B(Z)$ such that $x_{[0, m-1]} v u w x_{[0, m-1]} \in \B(Z)$. Then the configuration $z = x_{(-\infty,m-1]} v u w x_{[0, \infty)}$ is in $Z$, and thus in $\Ker{f}$. Since $u \notin \B(\Delta_X)$, the images of $z$ under the two projections from $\Ker{f}$ to $X$ differ in finitely many coordinates, so $f$ is not preinjective.

Conversely, suppose we have $f(x) = f(y)$ for some $x, y \in X$ which differ in finitely many coordinates. Then there exists a word $w u v \sqsubset (x,y)$ such that $u \notin \B(\Delta_X)$ but $v, w \in \B_m(\Delta_X)$. The set $Z'$ of configurations $z \in X^2$ where each coordinate $i \in \Z$ either satisfies $z_{[i-m,i+m]} \in \B(\Delta_X)$ or is part of an occurrence of $w u v$ is a subSFT of $\Ker{f}$. Since $\Delta_X$ is transitive, so is $Z'$, and thus $\Delta_X \subsetneq Z' \subset Z$.
\end{proof}

In the sofic case, it is easy to find a counterexample.

\begin{example}
Let $X = \B^{-1}((0^*(10^*2 + 30^*4))^*)$, which is a mixing sofic shift. Let $f : X \to \{0, 1, 2, 4\}^\Z$ be defined by $f(x)_n = 1$ if $x_n = 3$, and $f(x)_n = x_n$ otherwise. It is easy to check that the only transitive component of $\Ker{f}$ is $\Delta_X$, but $f$ is not preinjective since $f(\INF 010\INF) = f(\INF 0 3 0\INF)$.
\end{example}

\subsection{Tools from Symbolic Dynamics}

Finally, we list some well known classical results of symbolic dynamics that we use repeatedly in the course of this article.

\begin{definition}
For two subshifts $X$ and $Y$, we denote $X \pergeq Y$ if the period of every periodic point of $X$ is divisible by the period of some periodic point of $Y$.
\end{definition}

The following result, taken from \cite{Bo83}, is very useful.

\begin{lemma}[Extension Theorem]
Let $f : X \to Y$ be a block map with $X \subset Z$, where $Y$ is a mixing SFT and $Z \pergeq Y$. Then there exists a block map $\tilde f : Z \to Y$ such that $\tilde f |_X = f$.
\end{lemma}

Taking $X = \emptyset$, we obtain the following corollary, since the condition $Z \pergeq Y$ is clearly necessary for the existence of a block map from $Z$ to $Y$. A similar combinatorial characterization for the existence of a block map between two general SFTs can be found in \cite{BaDy07}.

\begin{corollary}
\label{cor:ExistsMap}
If $Y$ is a mixing SFT and $Z$ any subshift, then there exists a block map from $Z$ to $Y$ if and only if $Z \pergeq Y$.
\end{corollary}

A further corollary (and a special case of the results of \cite{BaDy07}) is that this notion is decidable.

\begin{corollary}
\label{cor:PericDecidable}
Given two mixing SFTs $X$ and $Y$, it is decidable whether $X \pergeq Y$.
\end{corollary}

\begin{proof}
If we have $X \pergeq Y$, then there exists a block map from $X$ to $Y$ by Corollary~\ref{cor:ExistsMap}. On the other hand, if $X \pergeq Y$ does not hold, then there exists a periodic point $x \in X$ of some period $p \in \N$ such that $Y$ has no $d$-periodic points for any divisor $d$ of $p$. Thus $X \pergeq Y$ can be decided by enumerating the block maps from $X$ to $Y$, and the periodic points of $X$ and $Y$.
\end{proof}

Another useful result from \cite{Bo83} is the following.

\begin{lemma}[Lower Entropy Factor Theorem]
Let $X$ and $Y$ be mixing SFTs with $h(X) > h(Y)$. Then there exists a factor map from $X$ onto $Y$ if and only if $X \pergeq Y$.
\end{lemma}

The following result can be extracted from the proof of Lemma 10.1.8 in \cite{LiMa95}.

\begin{lemma}[Marker Lemma]
\label{lem:MarkerLemma}
Let $X \subset S^\Z$ be a shift space, and let $n \geq 1$. Then there exists a block map $h : X \to \{0, 1\}^\Z$ such that
\begin{itemize}
\item the radius of $h$ is at most $n^{|S|^{2n+1}}$,
\item the distance between any two $1$'s in $h(x)$ is at least $n$, and
\item if $h(x)_{(i-n, i+n)} = 0^{2n-1}$, then $x_{[i-n, i+n]}$ is $p$-periodic for some $p < n$.
\end{itemize}
\end{lemma}

Finally, we make use of the following Garden of Eden Theorem, a proof of which can be found in \cite{CeFiSc04}.

\begin{lemma}[Garden of Eden Theorem]
\label{lem:GoE}
Let $X \subset S^\Z$ be a transitive SFT, and $f : X \to X$ a cellular automaton. Then $f$ is surjective if and only if it is preinjective. In particular, if $f$ is injective, then it is bijective.
\end{lemma}

\section{The Symbolic Categories and their Morphisms}
\label{sec:Morphisms}

In this section, we define the thirteen categories that are the object of study of this paper, and study the categorical properties of their morphisms.

\subsection{The Categories}

\begin{definition}
We define a handful of categories of subshifts and block maps using the naming scheme $\mathfrak{R}n$, where $\mathfrak{R}$ denotes a restriction and $n$ a class of subshifts and block maps. For the properties, K stands for no restrictions, T for transitive subshifts, M for mixing subshifts, and P for mixing subshifts $X$ with a special uniform point $p(X)$, where each morphism $f : X \to Y$ must satisfy $f(p(X)) = p(Y)$. For the classes, 1 stands for all positive entropy SFTs and cellular automata on them (so that all morphisms are endomorphisms)\footnote{The reader may find it strange that we consider a category with multiple objects, even though there are no morphisms between them. Such readers, and others, may find is helpful to think of M1 as a category with just one object, say, a full shift: since category theoretical notions are defined in terms of morphisms, results about M1 would look roughly the same with this definition, just less canonical.}, 2 stands for all SFTs and block maps between them, and 3 for all sofic shifts and block maps between them. For example, M1 is the category of all mixing SFTs of positive entropy, and all cellular automata on them.

Finally, K4 is the category of all subshifts and all block maps between them.
\end{definition}

Between the categories thus defined, we have the following faithful inclusion functors, of which all but those from $\mathfrak{R}1$ to $\mathfrak{R}2$, from P$n$ to M$n$ and from K3 to K4 are also full.

\begin{center}
\begin{tikzpicture}[node distance=.75cm, auto]
  \node (P1) {P1};
  \node (P2) [below of=P1] {P2};
  \node (P3) [below of=P2] {P3};
  \node (U1) [right=1.5cm of P1] {M1};
  \node (U2) [below of=U1] {M2};
  \node (U3) [below of=U2] {M3};
  \node (M1) [right=1.5cm of U1] {T1};
  \node (M2) [below of=M1] {T2};
  \node (M3) [below of=M2] {T3};
  \node (K1) [right=1.5cm of M1] {K1};
  \node (K2) [below of=K1] {K2};
  \node (K3) [below of=K2] {K3};
  \node (K4) [right=1.5cm of K3] {K4};
  
  \draw [->] (P1) to (P2) {};
  \draw [->] (P2) to (P3) {};
  \draw [->] (U1) to (U2) {};
  \draw [->] (U2) to (U3) {};
  \draw [->] (M1) to (M2) {};
  \draw [->] (M2) to (M3) {};
  \draw [->] (K1) to (K2) {};
  \draw [->] (K2) to (K3) {};
  
  \draw [->] (P1) to (U1) {};
  \draw [->] (U1) to (M1) {};
  \draw [->] (M1) to (K1) {};
  \draw [->] (P2) to (U2) {};
  \draw [->] (U2) to (M2) {};
  \draw [->] (M2) to (K2) {};
  \draw [->] (P3) to (U3) {};
  \draw [->] (U3) to (M3) {};
  \draw [->] (M3) to (K3) {};

  \draw [->] (K3) to (K4) {};
\end{tikzpicture}
\end{center}

This choice of categories is motivated as follows. First, the category K4 contains all one-dimensional symbolic dynamics. The standard references \cite{LiMa95,Ki98} of symbolic dynamics take place mostly in K3, and this category is closed under all the standard operations (images, products, unions etc), and thus we mostly restrict our attention to these subshifts. The category K2 is as important as K3, and much easier to analyze. The transitive categories T(2/3) lack certain `pathological' objects, so many authors work exclusively on them, and much of their theory is known. For example, the existence of a factor map between two transitive SFTs, one of which has strictly more entropy than the other, has been given a simple characterization in \cite{Bo83}. The analogous result for the sofic case is claimed in \cite{Kr11}, although the condition is much more complicated. The mixing categories M(2/3) are somewhat similar to T(2/3), but there are certain key differences such as the Extension Lemma.

The main motivation for the endomorphism categories (K/T/M/P)1 are cellular automata, which in our formalism are endomorphisms of the full shift objects $S^\Z$ for all alphabets $S$. The pointed categories P(1/2/3) are generalizations of full shifts, and the addition of special uniform points is motivated by certain categorical constructions that it enables, and the fact that the existence of a fixed uniform point (a \emph{quiescent state}) is often assumed in the study of cellular automata.

An a posteriori motivation for having this many categories is that we find many subtle differences between them. For example, only the categories M1 and T1, as well as K3 and K4, have identical columns in Table~\ref{tab:Morphisms}, which summarizes the characterizations of different types of epic and monic morphisms.

We now begin to classify the morphisms of the symbolic categories. The characterizations we obtain are summarized in Table~\ref{tab:Morphisms}.

\begin{table}[htp]
\begin{center}
\caption{Known exact characterizations of properties of morphisms of the symbolic categories, and the legend for the abbreviations. See Definition~\ref{def:SPP} and Definition~\ref{def:PericQInj} for the nonstandard notions. The category of sets (assuming the axiom of choice) is given as a point of reference.}
\label{tab:Morphisms}

\vspace{0.5cm}

\begin{tabular}{|c|c|c|c|c|c|c|c|}
\hline
              & K4    & K3    & K2   & K1  & T3                     & T2  & T1  \\
\hline
epic          & sur   & sur   & sur  & sur & sur                    & sur & sur \\
split epic    &       &       & spp  & spp &                        & spp & bij \\
regular epic  & sur   & sur   & sur  &     &                        &     &     \\
monic         & inj   & inj   & inj  &     & ipp                    & inj &     \\
split monic   &       &       &      &     &                        &     & bij \\
regular monic & inj+s & inj+s & inj  &     & P~\ref{prop:RegMonics} & inj & bij \\
\hline
\hline
              & \textbf{Set} & M3                     & M2                    & M1  & P3                     & P2  & P1 \\
\hline
epic          & sur          & sur                    & sur                   & sur & sur                    & sur & sur \\
split epic    & sur          &                        & spp                   & bij &                        & spp & bij \\
regular epic  & sur          &                        &                       &     &                        &     & bij \\
monic         & inj          & P~\ref{prop:MonicsM3}  & P~\ref{prop:MonicsM2} &     &                        & pre & pre \\
split monic   & inj          &                        & inj+p                 & bij &                        & inj & bij \\
regular monic & inj          & P~\ref{prop:RegMonics} & inj                   & bij & P~\ref{prop:RegMonics} & inj & bij \\
\hline
\end{tabular}

\vspace{0.5cm}

\begin{tabular}{|c|c|c|c|}
\hline
sur & inj & bij & pre \\
\hline
surjective & injective & bijective & preinjective \\
\hline
\hline
ipp & s & p & spp \\
\hline
\parbox{2.35cm}{injective on periodic points} & subSFT image & peric & \parbox{2.35cm}{strong periodic point condition} \\
\hline
\end{tabular}
\end{center}
\end{table}

It is well known that the isomorphisms are always exactly the bijective block maps, but the categorical notions of injectivity and surjectivity are more subtle. The rest of this section is dedicated to characterizations of different flavors of epimorphisms and monomorphisms in the symbolic categories.

\subsection{Epimorphisms and Monomorphisms}
\label{sec:EpicMonic}

We begin with a study of epic and monic morphisms. The epic case is the simplest one, and introduces the reader to some of the basic arguments that we use in this paper. It is also essentially the same for all the categories. By contrast, the different classes of monomorphisms are the most varied, and this case shows some of the complications that can occur when one applied abstract categorical notions to concrete examples.

For morphisms of all concrete categories, split epic implies surjective implies epic, and split monic implies injective implies monic. The converses do not hold in general, but for epimorphisms we have the following.

\begin{proposition}
\label{prop:SurIsEpic}
In (K/T/M/P)(2/3) and K4, surjectivity is equal to epicness.
\end{proposition}

\begin{proof}
We only need to show non-surjective implies non-epic. So let $f : X \to Y$ be non-surjective. We need to show $f$ is not right-cancellative. Since $f(X) \subsetneq Y$, there is a word $w \in \B(Y)$ such that $w \notin \B(f(X))$. Let $g_0 : Y \to \{0, 1\}^\Z$ be the all-$0$ map, and let $g_1$ be the one induced by the characteristic function of $w$. Then $g_0 \circ f = (x \mapsto 0^\Z) = g_1 \circ f$, but $g_0 \neq g_1$, so $f$ is not epic.
\end{proof}

The following is proved analogously to Proposition~\ref{prop:SurIsEpic}, but some more technicalities are needed, since the only maps available are endomorphisms.

\begin{proposition}
\label{prop:SurIsEpicEndo}
In (K/T/M/P)1, surjectivity is equal to epicness.
\end{proposition}

\begin{proof}
Again, we only need to show that if $f : X \to X$ is not surjective, then it is not epic. For that, let $w \in \B(X)$ be such that $w \notin \B(f(X))$.

There are two cases, the first of which being that every extension of $w$ into a point of $X$ is periodic. This means that $Y = \{ x \in X \;|\; w \sqsubset x \}$ is a finite set of periodic points, in particular an SFT, $Z = X \setminus Y$ is also an SFT, and $f(X) \subset Z$. Define the block map $g : X \to X$ by $g|_Y = f|_Y$ and $g|_Z = \ID_Z$, so that $g(p(X)) = p(X)$ in the pointed case. Since $f(Y) \subset Z$, we have $g \neq \ID_X$, but $g \circ f = f$, so $f$ is not epic.

Suppose then that $w$ can be extended to a nonperiodic point of $X$. Then $\INF u w v \INF \in X$ is not periodic for some $u, v \in \B_n(X)$ and $n \in \N$. Let now $g : X \to X$ be the block map that behaves as the identity except for mapping the words $u^m w v^{m+1}$ to $u^{m+1} w v^m$ for some $m \in \N$ larger than the window size of $X$. Then we again have $g \circ f = f$, but clearly $g \neq \ID_X$.
\end{proof}

The next result is classical, but we include it for completeness.

\begin{lemma}
\label{lem:PeriodInj}
If a morphism $f : X \to Y$ of (T/M/P)2 is injective on periodic points, then it is injective.
\end{lemma}

\begin{proof}
Suppose that $f$ is not injective, so that we have $f(x) = f(y)$ for some $x, y \in X$ with $x_0 \neq y_0$. Now, suppose that there exist $i < 0$, $j > 0$ and $k \in \N$ with $k$ arbitrarily large such that $x_{[i-k,i]} = y_{[i-k,i]}$ and $x_{[j,j+k]} = y_{[j,j+k]}$. Let $k$ be larger than the window size of $X$ and the radius of $f$, and let $u = x_{[i-k,j+k]}$ and $v = y_{[i-k,j+k]}$. Since $X$ is transitive, there exists $w \in \B(X)$ such that $\INF (uw) \INF, \INF (vw) \INF \in X$. These points are distinct and periodic, but have the same $f$-image, and hence $f$ is not injective on periodic points.

On the other hand, if there is a bound for such $k$, then for arbitrarily large $n \in \N$, for arbitrarily large or small $i \in \Z$, the words $u = x_{[i,i+n]}$ and $v = y_{[i,i+n]}$ are such that $\INF u \INF, \INF v \INF \in X$. For either positive or negative $i$ and large enough $n$, these words are also distinct. Since $f(\INF u \INF) = f(\INF v \INF)$, we are done.
\end{proof}

This is what we know about monomorphisms.

\begin{proposition}
\label{prop:Monicness}
Let $f : X \to Y$ be a morphism in any of the categories.
\begin{itemize}
\item In (K/T/M/P)(1/2), if $f$ is monic, then it is preinjective. The converse holds in P(1/2/3).
\item If $f$ is injective, then it is monic. The converse holds in K(2/3/4) and T2.
\item In (T/M/P)3, if $f$ is injective on periodic points, then it is monic. The converse holds in T3.
\item In M(1/2/3), if $f$ is monic, then it is injective on uniform points.
\end{itemize}
\end{proposition}

\begin{proof}
Suppose first that $f$ is not preinjective in (K/T/M/P)(1/2), so that there exist $n \in \N$ larger than the window size of $X$ and words $u, v, v', w \in \B_n(X)$ such that $v \neq v'$ and $f(\INF u v w \INF) = f(\INF u v' w \INF)$. We may assume that $\INF u v w \INF$ is not periodic. Then the block map $g : X \to X$ that maps $u v w$ to $u v' w$ if no other $u v w$ overlaps it, and otherwise acts as the identity, is well-defined and nontrivial. Then $f \circ g = f \circ \ID_X$, and $f$ is not monic.

Suppose $f$ is preinjective in P(1/2/3), let $g \neq h : Z \to X$, and take $z \in Z$ with $g(z) \neq h(z)$. We may now assume that $z$ is asymptotic to $p(X)$, and then the points $g(z)$ and $h(z)$ are asymptotic. Since $f$ is preinjective, we have $f(g(z)) \neq f(h(z))$, and thus $f$ is monic.

It is clear that an injective morphism is always monic. Conversely, suppose $f$ is not injective in K(2/3/4), so that the kernel set $\Ker{f} \subset X \times X$ is strictly larger than the diagonal $\Delta_X$. This implies that the projection maps $p_1, p_2 : \Ker{f} \to X$ are distinct, but $f \circ p_1 = f \circ p_2$ by definition. Since $\Ker{f}$ is also an object of K(2/3/4), $f$ is not monic.

Then suppose $f$ is noninjective in T2, thus not injective on periodic points by Lemma~\ref{lem:PeriodInj}. Then $f$ equates two distinct periodic points $x, y \in X$ of period $n \in \N$. Let $Z$ be the orbit of $z = \INF (0^{n-1} 1) \INF$, and define $g, h : Z \to X$ by $g(z) = x$ and $h(z) = y$. Then $g \neq h$ but $f \circ g = f \circ h$, and $f$ is not monic.

Suppose next that $f$ is injective on periodic points in (T/M/P)3, and let $g \neq h : Z \to X$ be morphisms. Since $Z$ is transitive, there exists a periodic point $z \in Z$ such that $g(z) \neq h(z)$, but then $f(g(z)) \neq f(h(z))$, since the images of $z$ are periodic. As in the case of T2, we conversely see that if $f$ is not injective on periodic points, it is not monic in T3.

Suppose finally that the category is M(1/2/3), and suppose that $f$ is not injective on uniform points of $X$, so that $f(\INF a \INF) = f(\INF b \INF)$ for some $a \neq b \in B_1(X)$. Let $g_a$ and $g_b$ be the CA from $X$ to itself that send everything to $\INF a \INF$ and $\INF b \INF$, respectively. Then we have $f \circ g_a = f \circ g_b$, and $f$ is not monic.
\end{proof}

In T3, we show that Lemma~\ref{lem:PeriodInj} does not hold, and monicness does not imply preinjectivity.

\begin{example}
Let $X = \B^{-1}((0^+21^+2 + 0^+31^+3)^*)$, which is a mixing sofic shift with uniform points ($\INF 0 \INF$ and $\INF 1 \INF$), and let $f : X \to \{0,1,2,3\}^\Z$ be the block map that behaves as the identity except for sending each word $021$ and $031$ to $001$. Then $f$ is clearly not preinjective, but it is injective on periodic points, thus monic in (T/M/P)3.
\end{example}

The exact characterization of monomorphisms of M(2/3) has proved difficult. However, we obtain the following technical characterizations in terms of the kernel sets of block maps, which at least allows us to easily decide monicness.

\begin{proposition}
\label{prop:MonicsM3}
Let $f : X \to Y$ be a morphism of M3. Then $f$ is not monic if and only if $\Ker{f} \subset X^2$ has a mixing sofic subshift not contained in $\Delta_X$.
\end{proposition}

\begin{proof}
Suppose first that $f$ is not monic, so that there exist morphisms $g \neq h : Z \to X$ with $f \circ g = f \circ h$. Then the image of the block map $z \mapsto (g(z), h(z)) \in \Ker{f}$ is a mixing sofic subshift of $\Ker{f}$. Since $g \neq h$, this image is not contained in $\Delta_X$.

Suppose then that there exists such a subshift $Z \subset \Ker{f}$, which is then an object of M3. The restrictions of the projection maps $p_1, p_2 : \Ker{f} \to X$ to $Z$ now satisfy $f \circ p_1 = f \circ p_2$, and since $Z$ is not contained in $\Delta_X$, we also have $p_1 \neq p_2$.
\end{proof}

Note that we cannot replace the `mixing sofic subshift' above with `mixing component', since all transitive components of $\Ker{f}$ containing it might be nonmixing. The case of M2, however, is different.

\begin{lemma}
\label{lem:MixingInComponent}
A transitive SFT which contains a mixing subshift is itself mixing.
\end{lemma}

\begin{proposition}
\label{prop:MonicsM2}
Let $f : X \to Y$ be a morphism of M2. Then $f$ is not monic if and only if $\Ker{f} \subset X^2$ has a mixing component not equal to $\Delta_X$.
\end{proposition}

\begin{proof}
Suppose that $f$ is not monic. As in the previous proof, we see that $\Ker{f}$ has a mixing sofic subshift $Z$ not contained in $\Delta_X$. Let $m \in \N$ and $S$ be the window size and alphabet of $\Ker{f}$, and let $Z' \subset S^\Z$ be the SFT defined by the forbidden words $S^m \setminus \B_m(Z)$. Then $Z \subset Z' \subset \Ker{f}$, and by the transitivity of $Z$, for all $u, v \in \B_m(Z') = \B_m(Z)$ there exists a word $w \in \B(Z)$ such that $u w v \in \B(Z) \subset \B(Z')$. Since $m$ is a window size for $Z'$, this implies that $Z'$ is transitive. Then $Z'$ is contained in some transitive component $C$ of $\Ker{f}$, and $C$ is mixing by Lemma~\ref{lem:MixingInComponent}. Finally, since $Z \subset C$, we see that $C \neq \Delta_X$.

The converse case is proved exactly as above.
\end{proof}

By Proposition~\ref{prop:Monicness}, the monomorphisms of M2 are preinjective and injective on uniform points. The following example, which uses the above result, shows that neither this nor injectivity is a characterization.

\begin{example}
\label{ex:XORMonicness}
Let $X = \{0,1\}^\Z$, and let $f : X \to X$ be the three-neighbor XOR cellular automaton, defined by the local function $F(a, b, c) = a + b + c \bmod 2$. We show that $f$ is monic in M(2/3), even though it is not injective. Namely, one easily sees that $\Ker{f} = \{ (x, x+y) \;|\; x \in X, y \in Y \}$, where $Y = \{ \INF 0 \INF, \INF (011) \INF, \INF (101) \INF, \INF (110) \INF \}$, and the sums are taken cellwise. Then $\Ker{f}$ consists of two disjoint transitive components, of which exactly $\Delta_X$ is mixing. In the other transitive component, every point has period of $3$, and thus this component contains no mixing subshifts.

Let then $Z \subset \{0,1\}^\Z$ be defined by the forbidden words $\{ 000, 111 \}$, and let $g : Z \to \{0,1\}^\Z$ be the two-neighbor XOR automaton, with the local function $G(a, b) = a + b \bmod 2$. We show that $g$ is not monic in M(2/3), even though it is preinjective, and (vacuously) injective on unary points. Namely, $\Ker{g}$ consists of two transitive components, $\Delta_X$ and $\{ (x,x+\INF 1 \INF) \;|\; x \in X \}$, both of which are mixing.
\end{example}

\subsection{Split Epicness}
\label{sec:SplitEpi}

In this subsection, we show that split epicness is decidable in the category K3, and thus in the subcategories (M/T/K)(2/3), since the inclusion functors are full and faithful. An easy argument extends this result to all categories except K4, where decidability questions make little sense. In K2 and its subcategories, we give a concrete characterization in terms of periodic points. The case of (T/M/P)1 is rather trivial, see Proposition~\ref{prop:1Split} in the next section.

Recall that a morphism $f : X \to Y$ is split epic if and only if it has a section, that is, a morphism $g : Y \to X$ such that $f \circ g = \ID_Y$. For block maps, this is equivalent to the existence of a subshift $Z \subset X$ (the image of $g$) such that $f|_Z : Z \to Y$ is a conjugacy. Split epicness is a stronger version of the condition of having a \emph{cross section}, that is, a continuous (but not necessarily shift-commuting) map $h : Y \to X$ such that $f \circ h = \ID_X$. These notions are distinct: on the full shift, exactly the open maps have cross-sections \cite{He69}.

We begin the proof with a Ramsey theoretical lemma.

\begin{definition}
Let $k, p \in \N$. We write $r(k, p)$ for the least number $N \in \N$ such that if the edges of the size-$N$ complete graph $K_N$ are colored with $k$ colors, then there exists a monochromatic induced subgraph $G \subset K_N$ of size $p$.
\end{definition}

The fact that these numbers exist is (a special case of) the well-known Ramsey's theorem, which can be found in most standard references of combinatorics, including \cite{Ne95}.

\begin{lemma}
\label{lem:SemigroupRamsey}
Let $M$ be a finite monoid. Then there exists $k = q(M) \in \N$ such that for any $(a_1, \ldots, a_k) \in M^k$, there exist indices $i_1, i_2$ such that $a = a_{i_1} \cdots a_{i_2-1}$ is an idempotent, that is, $a = a^2$.
\end{lemma}

\begin{proof}
If $k \geq r(|M|, 3) - 1$, we can apply Ramsey's theorem to the complete graph with vertex set $[1, k+1]$ and edge coloring $\{i, j\} \mapsto a_i \cdots a_{j-1} \in M$ to obtain three distinct elements $\{i_1, i_2, i_3\}$ such that $a_{i_1} \cdots a_{i_2-1} = a_{i_2} \cdots a_{i_3-1} = a_{i_1} \cdots a_{i_3-1}$. Then
\[ a = a_{i_1} \cdots a_{i_2-1} = a_{i_1} \cdots a_{i_2-1} \cdot a_{i_2} \cdots a_{i_3-1} = a^2, \]
and we are done.
\end{proof}

Next, we prove a technical version of the Marker Lemma that involves the syntactic monoid of a sofic shift and an auxiliary monoid homomorphism.

\begin{definition}
Let $X \subset S^\Z$ be a sofic shift and $H : S^* \to M$ a homomorphism to a finite monoid $M$. If $w \in \B(X)$ is such that $(w^t)_X = (w)_X$ and $H(w^t) = H(w)$ for all $t \geq 1$, we say $w$ is \emph{pumpable for $H$}. For a monoid $M$, define $2^M$ as the monoid whose elements are subsets of $M$ and multiplication is defined by $A \cdot B = \{a \cdot b \;|\; a \in A, b \in B\}$.
\end{definition}

\begin{lemma}
\label{lem:RamseyPeriodMarker}
Let $X \subset S^\Z$ be a sofic shift and $H : S^* \to M$ a homomorphism to a finite monoid $M$, and define $K = q(\Syn(X) \times M)$. Then there exists a block map $h : X \to \{0, 1\}^\Z$ with radius at most $3K^2 + K^{|S|^{2K+1}}$ such that the following properties hold for all $x \in X$, $n \in \N$ and $k = 2K^2 + 1$:
\begin{itemize}
\item if $h(x)_{[0,n-1]} = 1^n$, then $n \leq k$,
\item if $h(x)_{[0,n+1]} = 01^n0$, then $w = x_{[1,n]}$ is pumpable for $H$,
\item if $h(x)_{[0,2k-1]} = 0^{2k}$, then $x_{[0,2k-1]}$ is periodic with period $p \leq k$,
\item if $h(x)_{[0,n]} = 01^\ell0^{2k}$ ($h(x)_{[0,n]} = 0^{2k}1^\ell0$), then $x_{[1,n]}$ ($x_{[0,n-1]}$, respectively) is periodic with period $p \leq k$, and $p$ divides $\ell$.
\end{itemize}
\end{lemma}

\begin{proof}
First, let $h_K : X \to \{0,1\}^\Z$ be given by the Marker Lemma for the constant $K$. Now, define $h$ as follows for a configuration $x \in X$.

If $h_K(x)_{[-\ell,r+1]} = 10^{r+\ell}1$ for some $r, \ell \geq 0$ such that $r + \ell < 3K^2$, then $r + \ell \geq K$. By the definition of $K$, there then exist $i, j \in \N$ with $j > 0$ and $i + j \leq \ell + r$ such that $x_{[-\ell + i,-\ell + i + j]}$ is pumpable for $H$. We choose the minimal such $i$ and $j$, and define $h(x)_{[-\ell,r]} = 0^i 1^j 0^{r+\ell+1-i-j}$.

If $h_K(x)_{[0,3K^2]} = 10^{3K^2}$, then $w = x_{[1,3K^2]}$ is periodic with some period $p \leq K$, so that $w = u^m v$ for some $m \geq 3K$ and $u, v \in \B(X)$ with $|u| = p$ and $|v| < p$. Now, there exist $i, j \in \N$ with $j > 0$ and $i + j \leq K$ such that $x_{[ip, (i + j)p]}$ is pumpable for $H$. We again choose the minimal such $i$ and $j$, and define $h(x)_{[0,2K^2]} = 0^{1+ip} 1^{jp} 0^{2K^2-(i+j)p-1}$. The case for $h_K(x)_{[0,3K^2]} = 0^{3K^2}1$ is handled symmetrically.

Finally, if $h_K(x)_{[-K^2,K^2]} = 0^{2K^2+1}$, then by the definition of $h_K$, $x_{[-K^2,K^2]}$ is periodic with period $p \leq K$. In this case, we define $h(x)_{[-K^2,K^2]} = 0^{2K^2+1}$. We have now defined the block map $h$ completely, and the desired properties follow.
\end{proof}

\begin{definition}
\label{def:SPP}
Let $f : X \to Y$ be a block map, and define
\[ \mathcal{P}_p(Y) = \{ u \in \B(Y) \;|\; \INF u \INF \in Y, |u| \leq p \}. \]
We say $f$ satisfies the \emph{strong $p$-periodic point condition} if there exists a length-preserving function $G : \mathcal{P}_p(Y) \to \B(X)$ such that for all $u, v \in \mathcal{P}_p(Y)$ and $w \in \B(Y)$ with $\INF u .w v \INF \in Y$, there exists an $f$-preimage for $\INF u .w v \INF$ of the form
$\INF G(u) w' . w'' w''' G(v) \INF \in X$
where $|u|$ divides $|w'|$, $|v|$ divides $|w'''|$ and $|w| = |w''|$. The \emph{strong periodic point condition} is that the strong $p$-periodic point condition holds for all $p \in \N$.
\end{definition}

We are now ready to prove the main result of this section.

\begin{theorem}
\label{thm:SplitEpicSolved}
Given two objects $X \subset S^\Z$ and $Y \subset R^\Z$ and a morphism $f : X \to Y$ in K3, it is decidable whether $f$ is split epic. If $X$ is an SFT, split epicness is equivalent to the strong periodic point condition.
\end{theorem}

\begin{proof}
We assume that $f$ is a symbol map by recoding $X$ if necessary, and assume that it has a section $g : Y \to X$. Let $M = 2^N$ where $N$ is the syntactic monoid of $X$. Define the map $H : R^* \to M$ by $w \mapsto \{ (u)_X \;|\; u \in f^{-1}(w) \}$. It is easy to see that $H$ is a monoid homomorphism. Let $h : Y \to \{0,1\}^\Z$ and $k \in \N$ be given by Lemma~\ref{lem:RamseyPeriodMarker} for $Y$ and $H$. We now construct another section $\phi : Y \to X$ with radius at most $3(k + K^2) + K^{|S|^{2K+1}}$, where $K$ is as in Lemma~\ref{lem:RamseyPeriodMarker}. Let $G : \bigcup_{p \in \N} \mathcal{P}_p(Y) \to \B(X)$ be the function $u \mapsto g(\INF u . u \INF)_{[0,|u|-1]}$.

Let $y \in Y$. We now give names to certain subwords of $y$ to simplify the discussion that follows. If $h(y)_{[i-1,j+1]} = 0 1^\ell 0$, then $y_{[i,j]}$ is a \emph{marked pumpable word}. Suppose then that $h(y)_{[i-1,j+1]} = 1 0^\ell 1$. If $\ell \leq 2k$, then $y_{[i,j]}$ is a \emph{short words}, and otherwise (including the cases where $-i$ and/or $j$ is infinite) a \emph{long periodic word}. We proceed by defining the $\phi$-images ($f$-preimages) for these words in the following order:
\begin{enumerate}
\item short words,
\item long periodic words, and
\item marked pumpable words.
\end{enumerate}

Now, the idea is to think of a marked pumpable word as being an infinite repetition of that word, so that local rules cannot `see' beyond such a repetition. More precisely, consider a subword $v_1 w v_2$ of $y \in Y$ where the $v_i$ are marked pumpable words and $w$ is a short word. Consider the point $y' = \INF v_1 .w v_2 \INF \in Y$, for which we have $g(y') = \INF G(v_1) w'. w'' w''' G(v_2) \INF \in X$, where $|w''| = |w|$. The local rule of $\phi$ chooses $w''$ as the $f$-preimage of $w$, and then $w''$ only depends on the word $v_1 w v_2$.

Consider then a long periodic word $w$ in $y$. By the properties of $h$, it is actually periodic with some period $p \leq k$, so denote $w = u^\ell u'$, where $|u| = p$, $|u'| < p$ and $\ell \geq 2$. In principle, $\ell$ may also be infinite, but it is enough to consider finite $\ell$, and handle the infinite case by taking the limit of the finite cases. The local rule of $\phi$ chooses $G(u)^\ell G(u)_{[0,|u'|-1]}$ as the $f$-preimage of $w$, and this can be computed locally with a radius of $k$.

Consider finally a subword $v_1 w_1 v_2 w_2 v_3$ of $y$, where the $w_i$ are either short words or long periodic words, and the $v_i$ are marked pumpable words. Again, the length of $w_1$ or $w_2$ may be infinite, but we handle this case by taking the limit. Now, we have already chosen $f$-preimages $w_i''$ for the $w_i$ such that
\[ g(\INF v_i .w_i v_{i+1}^\infty) = \INF G(v_i) w_i'. w_i'' w_i''' G(v_{i+1}) \INF \in X \]
for some $w_i', w_i''' \in \B(X)$ of minimal length. If $w_i$ is a short word, this follows directly from the way $w_i''$ was chosen, and if $w_i$ is a long periodic word, this follows by noting that $\INF v_i w_i v_{i+1} \INF$ is actually periodic with period dividing both $|v_i|$ and $|v_{i+1}|$, and choosing $w_i' = w_i''' = \epsilon$. In particular, $|v_2|$ divides $|w_1'''|$ and $|w_2'|$.

Now, suppose $g$ has radius $r \in \N$. Then $|w_i'|, |w_i'''| \leq r$, and
\[ \INF v_1 .w_1 v_2^{2r+1} w_2 v_3 \INF \in Y \]
because $v_2$ is pumpable. We also have
\begin{align*}
g(\INF v_1 .w_1 v_2^{2r+1} w_2 v_3 \INF) = \INF G(v_1) w_1'. w_1'' w_1''' G(v_2)^\ell w_2' w_2'' w_2''' G(v_3) \INF,
\end{align*}
for some $\ell > 0$. Since $v_2$ is pumpable for $H$, we have $H(v_2) = H(v_2^{2r+1})$, and in particular there exists an $f$-preimage $v_2' \in \B(X)$ of $v_2$ such that $v_2' \sim_X w_1''' G(v_2)^\ell w_2'$. The local rule of $\phi$ chooses such a $v_2'$ as the $f$-preimage of $v_2$. If both $w_1$ and $w_2$ are short words, then $v_2'$ depends only on the subword $v_1w_1v_2w_2v_3$. If $w_1$ ($w_2$) is a long periodic word, then $v_2'$ is determined by $v_2w_2v_3$ ($v_1w_1v_2$, respectively). By the definition of $h$, at least one of the words must be a short word.

We have now defined a block map $\phi : Y \to X$ which is clearly a section of $f$, since $\phi(x)_i$ was chosen as an $f$-preimage of $x_i$ for all $x \in X$ and $i \in \Z$. The radius of $\phi$ is at most $3k$ plus the radius of $h$.

Suppose finally that $X$ is an SFT with window size $2$. If $f$ is split epic, then clearly the $p$-periodic point condition holds for all $p \in \N$, as a section of $f$ gives a consistent set of preimages for each set of periodic points. Conversely, if the strong periodic point condition holds, we proceed as above, but define the map $G$ using the $k$-periodic point condition, as the values $G(u)$ were only needed in the proof when $|u| \leq k$. Instead of choosing the preimages of $\INF v_i w_i v_{i+1}^\infty$ using an assumed section, we use the strong periodic point condition, and since $X$ is an SFT with a small window size, the preimages can safely be glued together.
\end{proof}

It is easy to find examples of block maps between (mixing) sofic shifts that satisfy the strong periodic point condition, but are not split epic. Thus the characterization cannot be extended to K3 or even M3.

\begin{example}
Let $X = \B^{-1}((0^*10^*20^*3)^*)$ and $Y = \B^{-1}((0^*10^*10^*3)^*)$, and define the symbol map $f : X \to Y$ by $2 \mapsto 1$ and $a \mapsto a$ for $a \in \{0,1,3\}$. Then $f$ is surjective and satisfies the strong periodic point condition (it is even injective on periodic points, which is a stronger condition), but is not split epic.
\end{example}

The simpler condition that every periodic point of $Y$ have a preimage with the same period is not sufficient for split epicness even in the case of SFTs, as shown by the following example.

\begin{example}
Let $S = \{0,1,\#\}$, and define the mixing SFTs $X, Y \subset \{0,1,\#\}^\Z$ by $X = \B^{-1}(((0^+ + 1^+)\#)^*)$, and $Y = \B^{-1}((\#^+(0 + 1))^*)$. Define the block map $f : X \to Y$ by the local function
\[ F(a,b) =
	\left\{ \begin{array}{ll}
		b, & \mbox{if } a = \#, \\
		\#, & \mbox{otherwise.}
	\end{array} \right. \]
Intuitively, configurations of $X$ consist of arbitrarily long runs of $0$s and $1$s separated by the $\#$-symbols, and $f$ compresses these runs into single symbols in $Y$. The morphism $f$ is surjective and every periodic point has a preimage of the same period, but is not split epic in any of the categories.
\end{example}

We mention the following interesting property of split epic morphisms.

\begin{proposition}
Assume $f : X \to Y$ is split epic in any of the categories, where $X$ is a mixing SFT. Then $Y$ is also a mixing SFT.
\end{proposition}

\begin{proof}
Since $f$ is surjective, we have $f(X) = Y$, and this subshift is mixing sofic. Let $g : f(X) \to X$ be a section of $f$. Then $g \circ f$ is an idempotent cellular automaton on $X$, so $g(f(X)) \subset X$ is a mixing SFT by \cite{Ta07}. Since $g$ is an isomorphism between $f(X)$ and $g(f(X))$, also $f(X) = Y$ is a mixing SFT.
\end{proof}

Suppose that $f : X \to Y$ is split epic, where $X$ is a mixing SFT and $Y \subset X$, so that $f$ can be seen as a cellular automaton on $X$. By the above, the image $Y$ is a mixing SFT, and it is tempting to ask whether this holds for the limit set of $f$, which would imply that $f$ is stable. However, we have the following counterexample. Let $X = \{0, 1, \hat 0, \hat 1\}^\Z$, and let $g$ be any unstable cellular automaton on $Y = \{0, 1\}^\Z$. Define $f : X \to Y$ by stating $f|_Y = g$ and $f(\hat x) = x$ for all $x \in Y$, and then using the Extension Lemma to extend $f$ to the whole of $X$. Now, the morphism $h : Y \to X$ defined by $h(x) = \hat x$ is a section for $f$, so that $f$ is split epic. However, $f$ is unstable since $g$ is.

\subsection{Other Classes of Morphisms}

We now discuss split monicness, and the simpler cases of split epicness not covered by the previous subsection.

\begin{proposition}
\label{prop:1Split}
In (T/M/P)1, all split epis and split monos are isomorphisms.
\end{proposition}

\begin{proof}
If $f : X \to X$ has a retract, then $f$ is injective, hence bijective by the Garden of Eden Theorem. Conversely, if $f$ has a section $g$, then $g$ is injective, hence bijective, and $f$ is its inverse.
\end{proof}

In K1, there are split epimorphisms and split monomorphisms that are not bijective.

\begin{example}
Let $X = \B^{-1}(0^*1^*2^*)$, and define  $f : X \to X$ by $f(\INF 0 . 1^n 2 \INF) = \INF 0 . 1^{n-1} 2 \INF$ for all $n \geq 1$, and $f(\INF 0 . 2 \INF) = \INF 0 . 2 \INF$. Now $f$ is split epic, since the block map $g : X \to X$ defined by $g(\INF 0 . 1^n 2 \INF) = \INF 0 . 1^{n+1} 2 \INF$ for all $n \geq 0$ is its section. Similarly, $g$ is split monic, since $f$ is its retract. These maps are not bijective.
\end{example}

For split monomorphisms, the characterization in the mixing SFT case is basically just the Extension Lemma. Note that the usual proof of the Extension Lemma is similar to our arguments in the previous subsection, which is not surprising since the two are dual concepts.

\begin{definition}
\label{def:PericQInj}
We say a block map $f : X \to Y$ is \emph{peric} if $X \perleq Y$.
\end{definition}

Note that the condition of being peric is not really a property of block maps, but instead a property of pairs of subshifts. Also, every morphism is peric in P(1/2/3) and (K/T/M/P)1. 

\begin{proposition}
The split monics of (M/P)2 are exactly the peric injections.
\end{proposition}

\begin{proof}
If $f : X \to Y$ has a retract, then necessarily $X \perleq Y$ by Corollary~\ref{cor:ExistsMap}, and $f$ is injective. Conversely, if $X \perleq Y$ holds and $f$ is injective, then $f(X)$ is a mixing SFT conjugate to $X$ via $f$. Let $g : f(X) \to X$ be the inverse of $f : X \to f(X)$. By the Extension Lemma, $g$ has an extension $\tilde g : Y \to X$, which is then a retract of $f$.
\end{proof}

In particular, the split monics of P2 are exactly the injections.

\begin{corollary}
It is decidable whether $f : X \to Y$ is split monic in (M/P)2.
\end{corollary}

\begin{proof}
Pericness is decidable by Corollary~\ref{cor:PericDecidable}. The decidability of injectivity is standard.
\end{proof}

We finally consider regular epimorphisms and monomorphisms. Recall that an epimorphism (monomorphism) is regular if it is the coequalizer (equalizer) of some pair of morphisms. The case of regular epis in K(2/3/4) is a byproduct of the proof of Proposition~\ref{prop:Regularity} in Section~\ref{sec:Properties} (see Section~\ref{sec:Limits} for more on equalizers and coequalizers).

\begin{proposition}
\label{prop:RegularEpis}
In K(2/3/4), every epimorphism is regular, and thus the regular epimorphisms are exactly the surjections.
\end{proposition}

For regular monomorphisms, however, the situation is different. The following result uses Proposition~\ref{prop:LimitsExist} and Proposition~\ref{prop:Equalizers} from Section~\ref{sec:Limits}.

\begin{proposition}
\label{prop:RegMonics}
In (K/T/M/P)2 and K(3/4), a monomorphism $f : X \to Y$ is regular if and only if it is injective and $f(X)$ is a subSFT of $Y$. In (M/P)3 (T3), it is regular if and only if it is injective and $f(X)$ is the unique maximal mixing (transitive, respectively) sofic subshift of some subSFT of $Y$.
\end{proposition}

\begin{proof}
First, every equalizer in the aforementioned categories is of the corresponding form, by the two propositions.

Conversely, up to composition with an isomorphism, every injective map is an inclusion $i : X \hookrightarrow Y$. Consider first the categories (K/T/M/P)2 and K(3/4), and suppose $X$ is a subSFT of $Y$. Let $m \in \N$ be the window size of $X$ relative to $Y$, and define two block maps $g_0, g_X : Y \to \{0, 1\}^\Z$ by $g_0(x) = \INF 0 \INF$ for all $x \in X$, and $g_X(x)_0 = 0$ if and only if $x_{[0, m-1]} \in \B_m(X)$. Then, $g_X^{-1}(\INF 0 \INF) = X$, the block maps are morphisms of the same category as $i$, and $i$ is the equalizer of $g_0$ and $g_X$ in that category, by Proposition~\ref{prop:LimitsExist} and Proposition~\ref{prop:Equalizers}.

Next, consider the categories (M/P)3 (T3), and suppose that there exists a subSFT $Z \subset Y$ such that $X$ is the unique maximal mixing (transitive, respectively) sofic subshift of $Z$. Using the above construction to obtain $g_0, g_Z : Y \to \{0,1\}^\Z$, we have $Z = g_Z^{-1}(\INF 0 \INF)$, and then $i$ is the equalizer of $g_0$ and $g_Z$ by Proposition~\ref{prop:Equalizers}.
\end{proof}

In particular, a monomorphism of (K/T/M/P)2 is regular if and only if it is injective. Since all monomorphisms of K2 are thus regular, but those of K3 are not, we can finally state the obvious.

\begin{corollary}
The categories K2 and K3 are not equivalent.
\end{corollary}

Regular monomorphisms are used in Section~\ref{sec:Other} to characterize the SFT objects of K(3/4). In the case of (T/M/P)3, however, we do not have a more exact characterization, and the following example shows some of the related complications.

\begin{example}
Let $X \subset \{0,1\}^\Z$ be the sofic shift consisting of those $x \in \{0,1\}^\Z$ that satisfy the following parity condition: whenever $x_{[i,i+2k+1]} = 0 1^{2k} 0$ and $x_{[j,j+2m+1]} = 0 1^{2m} 0$ for some $i, j \in \Z$ and $k, m > 0$, then $i = j$. Then $X$ contains a maximal transitive sofic shift $Y \subset X$, namely the one where every run of $1$s is of odd length, and $Y$ is also mixing. However, $Y$ is not a subSFT of $X$, since every subSFT of $X$ that contains $Y$ also contains every word $0 1^{2k} 0$ for large enough $k \in \N$. Next, let $Z \subset \{0,1,2\}^\Z$ be the mixing sofic shift defined by exactly the minimal-length forbidden words of $X$. The points of $Z$ are of the form $\cdots 2 w_{-1} 2 w_0 2 w_1 2 \cdots$ where $w_i \in \B(X)$ are arbitrary, possible empty or infinite. Then $X$ is clearly a subSFT of $Z$, obtained by forbidding the letter $2$. Now $Y$ is not a subSFT of $Z$, but by Proposition~\ref{prop:RegMonics}, the inclusion map of $Y$ to $Z$ is regular monic in the categories (T/M/P)3.

Conversely, not all inclusions of mixing sofic shifts are regular in (T/M/P)3. In fact, if $X$ is a mixing SFT and $Y \subset X$ a mixing proper sofic subshift, then the inclusion $i : Y \hookrightarrow X$ is not regular. For this, let $Z \subset X$ be a subSFT of $X$ containing $Y$. Then $Z$ is an SFT and $Y$ is contained in some transitive component $C$ of $Z$, which must then be mixing by Lemma~\ref{lem:MixingInComponent}. Since $C$ is also an SFT, it cannot equal $Y$, and thus $Y$ is not a maximal mixing or transitive sofic subshift of $Z$.
\end{example}

We also consider regular epi- and monomorphisms in the endomorphism categories (T/M/P)1.

\begin{proposition}
Every regular monic of (T/M/P)1 is an isomorphism. In P1, regular epimorphisms are isomorphisms.
\end{proposition}

\begin{proof}
First, let $f : X \to X$ be the equalizer of a pair $g, h : X \to X$ in (T/M/P)1. Since $f$ is monic, it is preinjective, and thus surjective. Since $g \circ f = h \circ f$, this implies $g = h$. Consider the identity morphism $\ID_X : X \to X$. By the definition of an equalizer, there exists a morphism $u : X \to X$ such that $\ID_X = f \circ u$, but then $f$ is injective, and thus bijective by the Garden of Eden Theorem.

Second, let $f$ be the coequalizer of $g$ and $h$ in the category P1, so that $f$ is in particular surjective. If we had $g(x) \neq h(x)$ for some $x \in X$, we could assume that $x$ is asymptotic to $p(X)$. Since $f \circ g = f \circ h$, $f$ would then not be preinjective, a contradiction with its surjectivity. Thus $g = h$, and again by considering the identity morphism we see that $f$ is bijective.
\end{proof}

On the object $\{0,1\}^\Z$ of the categories (T/M)1, the two-neighbor XOR automaton from Example~\ref{ex:XORMonicness} is the coequalizer of the identity automaton and the flip automaton (by a simple application of Corollary~\ref{cor:ContConnector}), so the above result for regular epis does not hold in these cases.

\section{Categoricity of Symbolic Dynamical Properties}
\label{sec:Other}

\subsection{Categorical Properties}

We say that a property of objects or morphisms of a category $\cat$ is \emph{categorical} if it only depends on the categorical structure of $\cat$, or in other words, it is invariant under isomorphism of categories. In the previous sections, we took standard categorical notions such as epicness and monicness, and investigated what they mean in the symbolic categories. The converse question is perhaps more interesting: Given a property of interest in the world of symbolic dynamics (say, the surjectivity of a morphism, or the SFTness of a sofic shift), and a symbolic category (say, K3), does the property correspond to some categorical notion? For example, in all the symbolic categories, the answer is yes for surjectivity: the surjective block maps are the epimorphisms. Injectivity is also categorical in most of our categories (note that in the categories T1, M1 and P1, bijectivity is equal to injectivity), and corresponds to either monicness or regular monicness.

In this section, we identify particular types of objects and morphisms using first-order formulas over the natural language of the category (which obviously implies that the notion is categorical). In addition to surjectivity and injectivity, Table~\ref{tab:Morphisms} contains several more exotic properties of block maps that we have shown to be categorical, like being preinjective in P2, or being the embedding of a subSFT in K3. Of course, not all properties of block maps are categorical, and we show here one example, namely right resolvingness. Recall that a block map $f : X \to Y$ is right resolving if for all $x \neq y \in X$ such that $x_i = y_i$ for all $i \leq 0$, we have $f(x) \neq f(y)$. Left resolvingness is defined analogously.

\begin{example}
The two-neighbor XOR automaton is both left and right resolving, while the modified two-neighbor XOR automaton $f : \{0,1,2\}^\Z \to \{0,1,2\}^\Z$ defined by the local function $(2,b) \mapsto 2$ and $(a,b) \mapsto a + b \bmod 2$ for all $a \in \{0,1\}$ and $b \in \{0,1,2\}$ is left resolving, but not right resolving.
\end{example}

To see why right and left resolvingness are not categorical properties, we define an automorphic functor of the symbolic categories, and show that it does not preserve these properties. This is the mirroring functor $R$, defined on individual configurations $x \in S^\Z$ by $x^R_i = x_{-i}$ for all $i \in \Z$, and then extended to subshifts by $X^R = \{ x^R \;|\; x \in X\}$ and to block maps $f : X \to Y$ by $f^R : X^R \to Y^R$ and $f^R(x^R) = f(x)^R$. It is obvious that $R$ is an involutive functor from any symbolic category to itself, thus its automorphism, and that $f^R$ is right resolving if and only if $f$ is left resolving. Since the notions are distinct, they cannot be categorical. However, we do not know whether the property of being left \emph{and} right resolving is categorical.

\subsection{Properties of Objects}

There are many other symbolic dynamical properties which are categorical. Consider the SFT objects of K(3/4). In symbolic dynamics, a subshift $X$ being an SFT is characterized by the condition that whenever $X_0 \supseteq X_1 \supseteq \cdots$ is an infinite decreasing sequence of subshifts with $\bigcap_{n \in \N} X_n = X$, there exists $n \in \N$ such that $X_n = X$ (because each of the finitely many forbidden words is already forbidden in $X_i$ for large enough $i$). In the language of category theory, the infinite intersection can be expressed as an inverse limit. We begin with the following lemma, which is valid in any category.

\begin{lemma}
\label{lem:InverseLimitIntersection}
If $X$, with the morphisms $j_n : X \to X_n$, is the inverse limit of the diagram $X_0 \stackrel{i_0}{\leftarrow} X_1 \stackrel{i_1}{\leftarrow} \cdots$ and each $i_n$ is monic, then each $j_n$ is also monic.
\end{lemma}

Note that such a diagram does not have an inverse limit in general.

\begin{proof}
We prove the claim by induction. First, suppose there exist $f \neq g : Y \to X$ such that $j_0 \circ f = j_0 \circ g$. Since $X$ is the inverse limit of the diagram, we have $i_0 \circ j_1 \circ f = j_0 \circ f = j_0 \circ g = i_0 \circ j_1 \circ g$, and since $i_0$ is monic, $j_1 \circ f = j_1 \circ g$. Inductively we get $j_n \circ f = j_n \circ g$ for all $n \in \N$, which is a contradiction with the fact that $f$ should be the unique morphism $h : Y \to X$ with $j_n \circ h = j_n \circ f$ for all $n \in \N$. The fact that the other $j_n$ are monic follows inductively, since $j_{n-1} = i_n \circ j_n$, and both $j_{n-1}$ and $i_n$ are monic.
\end{proof}

Consider the above situation in the categories K(3/4), and denote $Y_n = (i_0 \circ \cdots \circ i_{n-1})(X_n) \subset X_0$. There we actually have an infinite descending chain $Y_0 \supseteq Y_1 \supseteq Y_2 \supseteq \cdots \supseteq j_0(X)$, and we denote by $Y = \bigcap_{n \in \N} Y_n$ its intersection. In K4, the inverse limit property of $X$ implies that $j_0(X) = Y$, but in K3, it may still be that $j_0(X) \subsetneq Y$ if $Y$ is not an object of K3. For example, this is the case if $Y$ is a subshift without periodic configurations, each $Y_n$ is the SFT defined by the forbidden patterns of $Y$ of length at most $n$, and $X = \emptyset$ is the empty subshift. Thus we need some more restrictions. For a morphism $g : X \to Z$, we say that a sequence $(g_n : X_n \to Z)_{n \in \N}$ \emph{approximates} $g$ if $g = g_n \circ j_n$ holds for all $n \in \N$. 

\begin{lemma}
With the above notation, suppose that for all morphisms $g : X \to Z$ and all approximating sequences $(g_n)_{n \in \N}$ and $(g'_n)_{n \in \N}$ of $g$, we have $g_n = g'_n$ for some $n \in \N$. Then $j_0(X) = Y$.
\end{lemma}

\begin{proof}
Suppose first that $j_0(X) = Y$, and let $g_0|_Y = g'_0|_Y = g \circ j_0^{-1} : Y \to Z$ have radius $r \in \N$. Let $n \in \N$ be such that $\B_{2r+1}(Y_n) = \B_{2r+1}(Y)$, so that we have $g_0|_{Y_n} = g'_0|_{Y_n}$. This implies
\[ g_n = g_0|_{Y_n} \circ i_0 \circ \cdots \circ i_{n-1} = g'_0|_{Y_n} \circ i_0 \circ \cdots \circ i_{n-1} = g'_n. \]

Conversely, suppose that $j_0(X) \subsetneq Y$, and let $w \in \B(Y) \setminus \B(j_0(X))$. Denote by $g : X \to \{0,1\}^\Z$ the all-$0$ morphism. Then $g$ has two approximating sequences that differ for every $n \in \N$, namely, the sequence $(g_n)_{n \in \N}$ where every $g_n : X_n \to \{0,1\}^\Z$ is the all-$0$ map, and $(g'_n)_{n \in \N}$ where every $g'_n : X_n \to \{0,1\}^\Z$ is induced by the indicator function of $w$.
\end{proof}

Thus, the characterization of SFTs in K(3/4) is as follows.

\begin{proposition}
\label{prop:ThisIsSFT}
An object $X$ of K(3/4) is an SFT if and only if the following condition holds. If $(X_n)_{n \in \N}$ is a chain with monomorphisms $i_n : X_{n+1} \to X_n$ whose inverse limit is $X$, and every two approximating sequences of every morphism $g : X \to Z$ agree at some $n \in \N$, then there exists $n \in \N$ such that each $i_m$ for $m \geq n$ is an isomorphism. In K4, the condition on approximating sequences can be dropped.
\end{proposition}

This condition cannot be vacuously true unless $X = \emptyset$, since every nonempty subshift is an inverse limit of a diagram of SFTs. This way of characterizing the SFTs is nice, because it allows one to also extract the SFT objects of T3, M3 or P3 (with a similar proof). However, this is a `second order property', since we need to quantify over infinite diagrams. In K(3/4), SFTs $X$ can also be characterized with a first-order predicate, since the regular monomorphisms of K(3/4) are exactly those with subSFT images.

\begin{corollary}[of Proposition~\ref{prop:RegMonics}]
An object $X$ of K(2/3/4) is an SFT if and only if every monomorphism $e : X \to Y$ is regular.
\end{corollary}

Of these two characterizations of SFTs, the first seems more natural, since it captures the intuition that SFTs are `absolutely cofinite', while the characterization via regular monomorphisms sees SFTs as `absolutely equalizer-like' objects, and it characterizes SFTs in K(3/4) mainly because block maps are finitary by nature. We saw in Section~\ref{sec:Morphisms} that this does not characterize the SFTs in, for example, M3.

In the rest of this subsection, we characterize different properties of SFTs and sofic shifts in the categories K(2/3/4), starting with the following.

\begin{proposition}
Suppose $X$ is a finite subshift in K(2/3/4). Then being conjugate to $X$ is a first-order property.
\end{proposition}

\begin{proof}
First, suppose $X$ is the orbit of a single point. Then, $Y \cong X$ is equivalent to $Y$ not being a coproduct of two nonempty subshifts (since coproducts are disjoint unions in the categories), and $Y$ having the same number of endomorphisms as $X$. By induction, it is easy to construct a first-order statement in the case that $X$ is a coproduct of two smaller subshifts.
\end{proof}

We now study the boundaries of the Extension Lemma to show that being a mixing SFT is a categorical property in K(2/3/4).

\begin{lemma}
Suppose $Y$ is an SFT that is not mixing. Then there exist SFTs $X \subset Z$ such that the set of block maps from $Z$ to $Y$ is nonempty, and a block map $f : X \to Y$ that cannot be extended to a block map $\tilde f : Z \to Y$.
\end{lemma}

\begin{proof}
Suppose first that $Y$ is not even transitive, so that there exist $v_1, v_2 \in \B(Y)$ such that $v_1 w v_2 \notin \B(Y)$ for all $w \in \B(Y)$. Let $y_1, y_2 \in Y$ be eventually periodic points such that $v_i$ occurs in $y_i$, and let $q \in \N$ be the least common multiple of the eventual periods. Define $X$ as the union of the orbit closures of $x_1 = \INF (a_1 \#^{q-1})(b_1 \#^{q-1}) \INF$ and $x_2 = \INF (a_2 \#^{q-1})(b_2 \#^{q-1}) \INF$, and define the block map $f : X \to Y$ by $f(x_1) = y_1$ and $f(x_2) = y_2$. Next, define the SFT $Z$ by adding all points of the form
\begin{equation}
\label{eq:NewPoint}
\INF (a_1 \#^{q-1})(b_1 \#^{q-1})^j(a_2 \#^{q-1})^k(b_2 \#^{q-1}) \INF
\end{equation}
for $j, k \in \N$, and their orbit closures. There exists a block map from $Z$ to $Y$ which maps the whole of $Z$ onto a single periodic orbit whose period divides $q$, but $f$ cannot be extended to $Z$, as the images of points of the form~\eqref{eq:NewPoint} would contradict the assumption on $v_1$ and $v_2$.

Suppose then that $Y$ is transitive but not mixing. Let $p = \per(Y)$ be the period of $Y$, and let $\phi : Y \to \B^{-1}((0^{p-1} 1)^*)$ be the associated phase map with radius $r \in \N$. Let $v_1, v_2 \in \B_{2r + p}(Y)$ have different phases. Choose the points $y_i \in Y$ and $q \in \N$ as before, and define $X$, $f$ and $Z$ as above, but so that the words $v_i$ are aligned with the $b_i \#^{q-1}$-blocks of the preimages. Then the block map $f$ cannot be extended to $Z$, as the images of points of the form~\eqref{eq:NewPoint} would contradict the assumption that $v_1$ and $v_2$ have different phases.
\end{proof}

In view of the Extension Lemma, we have the following characterization.

\begin{corollary}
An SFT object $Y$ of K(2/3/4) is mixing if and only if the following holds: for all SFT objects $X$ and $Z$ such that there exists some morphism from $Z$ to $Y$, for all monomorphisms $i : X \to Z$ and all morphisms $f : X \to Y$, there exists a morphism $\tilde f : Z \to Y$ such that $f = \tilde f \circ i$.
\end{corollary}

Since mixing sofic shifts have mixing SFT covers, mixingness is a first-order property in K3. Using mixingness, we can then characterize transitivity.

\begin{lemma}
In K(2/3/4), for a sofic object $X$, the property of being a single periodic orbit is a first-order property.
\end{lemma}

\begin{proof}
This is equivalent to being a sofic shift without proper subshifts, which can be expressed for a sofic object $X$ of K(2/3/4) as `every monomorphism $f : Y \to X$ is an isomorphism'.
\end{proof}

Note that our first-order condition for being a single periodic orbit is different in K(2/3) and K4, since there exist minimal subshifts which are not periodic, and thus it is necessary to restrict to sofic objects in K4.

\begin{proposition}
\label{prop:TransitivityIsFO}
A sofic object $X$ of K(2/3/4) is transitive if and only if there exist a periodic orbit $Y$, a mixing SFT $Z$ and an epimorphism $f : Y \times Z \to X$. In particular, transitivity is a first-order property in K3.
\end{proposition}

\begin{proof}
Since every SFT satisfying the latter property is transitive, it suffices to prove the forward implication. For that, suppose $X$ is a transitive object of K(2/3). Since transitive sofic shifts have transitive SFT covers and a composition of epimorphisms is epic, we may assume that $X$ is an SFT. Let $p = \per(X)$ be the period of $X$, let $\phi : X \to \B^{-1}((0^{p-1} 1)^*)$ be the associated phase map with radius $r \in \N$, and let $m \in \N$ be a transition distance for $X$.

Recall that for all $u, v \in \B(X)$ there exist $k \in [0,p-1]$ and $w \in \B_{m+k}(X)$ such that $u w v \in \B(X)$. Let $\INF \hat w \INF \in X$ be any periodic point with $|\hat w| = q \geq \max(2r+p, m)$, and choose $Y = \B^{-1}((0^{q-1} 1)^*)$. Note that $p$ necessarily divides $q$. Construct the mixing SFT $Z$ by adding the new symbol $\#$ to $X$, and letting exactly the minimal length forbidden words of $X$ be forbidden in $Z$. We extend $\phi$ to a block map from $Z$ to itself by letting $\#$ be a spreading state.

Now, the surjective block map $f : Y \times Z \to X$ is constructed as follows. Let $(\INF (0^{q-1} 1) \INF, z) \in Y \times Z$, and define a \emph{block} to be an interval of the form $[\ell q, (\ell+1)q - 1] \subset \Z$. Now, a block $b$ is \emph{good} if $\# \not\sqsubset z_b$ and $\phi(z)_b = (0^{p-1}1)^{q/p}$, and otherwise \emph{bad}. A bad block is \emph{very bad} if both its neighbors are also bad. We define the $f$-image of each block, starting with the good blocks $b$, which are mapped to $z_b$. Also, each very bad block is mapped to the word $\hat w$. After this, $f$ can locally choose suitable images for the remaining bad blocks. Finally, we extend $f$ in a shift-invariant way and obtain a factor map from $Y \times Z$ to $X$.
\end{proof}

As a transitive subshift is a factor of the product of mixing and single-orbit subshifts, a nonwandering subshift is a factor of the product of mixing and finite subshifts.

\begin{lemma}
In K(2/3/4), a subshift $X$ is finite if and only if $X = f(Y^2)$ for some periodic orbit $Y$ and block map $f : Y^2 \to X$.
\end{lemma}

\begin{proof}
Suppose first that $X$ is finite, so that it is a disjoint union of some $k$ periodic orbits, with periods $p_0, \ldots, p_{k-1}$. Choose some representatives $x_0, \ldots, x_{k-1} \in X$ for the orbits. Let $p = \lcm \{ p_0, \cdots, p_{k-1} \}$, and let $Y$ be the orbit of $y = \INF (0^{k p - 1} 1) \INF$. Now, define $f : Y^2 \to X$ by mapping the configuration $(y, \sigma^{i + k j}(y))$ to $x_i$, for all $i \in [0, k-1]$ and $j \in [0, p-1]$. Since the least period of $x_i$ divides that of $(y, \sigma^{i + k j}(y))$, we can extend $f$ to $Y^2$ in a shift-invariant way, and then $X = f(Y^2)$. The converse direction is immediate, since $Y^2$ is finite.
\end{proof}

The following is proved almost exactly as the characterization of transitivity.

\begin{proposition}
A sofic object $X$ of K(2/3/4) is nonwandering if and only if there exist a finite subshift $Y$, a mixing SFT $Z$ and an epimorphism $f : Y \times Z \to X$. In particular, nonwanderingness is a first-order property in K3.
\end{proposition}

Using finiteness, we can also show that countability is a first-order property.

\begin{proposition}
A sofic object $X$ of K(2/3/4) is countable if and only if its every transitive sofic subobject is finite.
\end{proposition}

The properties of SFTness, mixingness and transitivity, together with the properties proved for morphisms, are quite versatile, and many other properties can be expressed using them.

\begin{example}
\label{ex:AFT}
A transitive object $X$ of K3 is \emph{of almost finite type} (AFT for short) if and only if there exists a surjection $f : Y \to X$, where $Y$ is an SFT, such that for any other SFT $Z$ and surjection $g : Z \to X$, there exists a unique $h : Z \to Y$ with $g = f \circ h$ \cite{BoKiMa85}. Thus being an AFT is a first-order categorical property in K3. There are several other definitions of AFTs in terms of their minimal right-resolving SFT covers, but as right-resolvingness is not a categorical notion, these are hard to express categorically.
\end{example}

\subsection{Variants and Invariants of Conjugacy}

Another interesting question is which of the well-known conjugacy invariants and weaker versions of conjugacy are categorical. We show that having the same entropy is a categorical property in K(2/3/4) and T2. First, we need a couple of lemmas.

\begin{lemma}[Proposition~4.4.6 in \cite{LiMa95}]
\label{lem:DecEntropy}
Let $(X_m)_{m \in \N}$ be a decreasing sequence of subshifts, that is, $X_{m+1} \subset X_m$ for all $m \in \N$, and denote $X = \bigcap_{m \in \N} X_m$. Then $\lim_m h(X_m) = h(X)$.
\end{lemma}

\begin{lemma}
\label{lem:Entropy}
For transitive SFTs $X$ and $Y$, we have $h(X) < h(Y)$ if and only if there exists a transitive SFT $Z$ such that $X \subsetneq Z$ and $Y$ factors onto $Z$.
\end{lemma}

\begin{proof}
If there exists such a $Z$, then $h(X) < h(Z) \leq h(Y)$, since a factor map decreases entropy, and a proper subshift of a transitive SFT has strictly lower entropy than it \cite[Corollary~4.4.9]{LiMa95}.

Suppose then that $h(X) < h(Y)$, and let $S$ be the alphabet of $X$. Similarly to the proof of Proposition~\ref{prop:TransitivityIsFO}, we define for every $m \in \N$ an SFT $Z_m$ over the alphabet $S \cup \{\#\}$ whose forbidden words are exactly the minimal length forbidden words of $X$, plus the set $\{ \# w \# \;|\; w \in S^*, 1 \leq |w| \leq m \}$. Then every $Z_m$ is mixing and satisfies $Z_{m+1} \subset Z_m$, and since $\bigcap_{m \in \N} Z_m = X \cup \{\INF \# \INF\}$, we also have $\lim_m h(Z_m) = h(X)$ by Lemma~\ref{lem:DecEntropy}. In particular, we have $h(Z_m) < h(Y)$ for some $m \in \N$. The claim then follows from the Lower Entropy Factor Theorem for $Z = Z_m$.
\end{proof}

\begin{lemma}[Theorem~4.4.4 in \cite{LiMa95}]
\label{lem:TransEntropy}
Let $X$ be an SFT. Then some transitive component $Y$ of $X$ satisfies $h(Y) = h(X)$.
\end{lemma}

\begin{proposition}
\label{prop:EntropyFO}
The property of having $h(X) < h(Y)$ (and thus also $h(X) = h(Y)$) for a pair of objects $X$ and $Y$ is first-order in K(2/3/4) and T2.
\end{proposition}

\begin{proof}
Suppose that we have $h(X) < h(Y)$. For all $m \in \N$, let $X_m$ be the SFT defined by the length-$m$ forbidden words of $X$. Then $(X_m)_{m \in \N}$ is a decreasing sequence of SFTs whose intersection is $X$, and Lemma~\ref{lem:DecEntropy} implies that $h(X_m) < h(Y)$ for some $m \in \N$. Then we have $h(X_m) < h(Z)$ for every SFT $Z$ into which $Y$ embeds. Thus $h(X) < h(Y)$ holds if and only if there exists an SFT $X'$ containing $X$ such that for all SFTs $Y'$ containing $Y$, we have $h(X') < h(Y')$. By Lemma~\ref{lem:Entropy} and Lemma~\ref{lem:TransEntropy}, this is a first-order condition.
\end{proof}

Of course, in the corresponding first-order formula, inclusions are replaced by monomorphisms (which corresponds to injective block maps in the four categories) and factor maps by epimorphisms.

Finally, we mention the \emph{zeta function} of a subshift, which encodes the number of periodic points of each period into a particular type of formal series. Two subshifts have the same zeta function if and only if they have the same number of periodic points of each period (or least period). This can be stated in categorical terms in the categories K(2/3/4), since the zeta function of a subshift is uniquely determined by the class of finite subshifts that can be embedded in it, and embeddings and finiteness are first-order categorical notions in K(2/3/4).

\section{Categorical Constructions}
\label{sec:Limits}

In this section, we study the existence and nature of standard categorical constructions, that is, limits and colimits, in the symbolic categories.

\subsection{Limits and Colimits}

We begin by establishing the existence of all finite limits in the categories K(2/3/4), describing the nature of these objects in the process. It is enough to prove the existence of terminal objects, binary products and equalizers, since all finite limits can be constructed from these.

\begin{proposition}
\label{prop:LimitsExist}
The categories K(2/3/4) are finitely complete, and finite limits in K(2/3) are computable.
\end{proposition}

\begin{proof}
A terminal object in these categories (in fact, in all of (K/T/M/P)(2/3) and K4) is the trivial subshift $T = \{\INF 0 \INF\}$, since for all objects $X$, there is a unique morphism $0_{XT} : X \to T$ that sends everything to the single element of $T$. The categorical product of two objects $X$ and $Y$ in (K/T/M/P)(2/3) and K4 is their coordinatewise product $X \times Y$ together with the projection symbol maps $p_1 : X \times Y \to X$ and $p_2 : X \times Y \to Y$, since every pair of block maps $f : Z \to X$ and $g : Z \to Y$ is uniquely factored through the projections by sending $z \in Z$ to the pair $(f(z),g(z)) \in X \times Y$. Finally, the equalizer of a parallel pair of morphisms $f, g : X \to Y$ in K(2/3/4) is simply the inclusion map $i$ of the subSFT $E = \{ x \in X \;|\; f(x) = g(x) \}$ of $X$, since every morphism $h : Z \to X$ with $f \circ h = g \circ h$ satisfies $h(Z) \subset E$, and thus factors uniquely through $i$.

From these constructions, and the fact that any finite limit can be mechanically constructed from finite products and equalizers \cite{Ma71}, it is clear that limits are computable in K(2/3).
\end{proof}

Now, it is known that the pullback of two morphisms $f : X \to Z$ and $g : Y \to Z$ is given by the equalizer of $f \circ p_1$ and $g \circ p_2$, where $p_1 : X \times Y \to X$ and $p_2 : X \times Y \to Y$ are the product projections. By the above, the pullback of $f$ and $g$ in K(2/3/4) is thus the fiber product $X \times_Z Y = \{ (x,y) \in X \times Y \;|\; f(x) = g(y) \}$, together with the projection maps to $X$ and $Y$. In particular, the kernel pair of $f : X \to Z$ is its kernel set $\Ker{f} = \{ (x,x') \in X^2 \;|\; f(x) = f(x') \}$.

Now, in the transitive categories (T/M/P)(2/3), the fiber product subshift $\{ x \in X \;|\; f(x) = g(x) \}$ defined for morphisms $f, g : X \to Y$ may not be an object, but if it is, then its inclusion into $X$ really is the categorical equalizer of $f$ and $g$. However, not all equalizers of T3 and (M/P)(2/3) are of this form, as shown in the following.

\begin{example}
Let $X = \{0,1\}^\Z$, and define $f : X \to X$ by $f(x)_0 = 0$ if and only if $x_{[0,2]} \in \{000,010,101\}$ for all $x \in X$. Let also $g : X \to X$ be the all-$0$ map. Now for $x \in X$ we have $f(x) = g(x)$ if and only if $x \in E = \{ \INF 0 \INF, \INF (01) \INF, \INF (10) \INF \}$. We claim that the equalizer of $f$ and $g$ in (M/P)(2/3) is the inclusion $i$ of $\{ \INF 0 \INF \}$ into $X$, which is not isomorphic to the equalizer of $f$ and $g$ in K(2/3).

Let $h : Y \to X$ be any morphism of (M/P)(2/3) such that $f \circ h = g \circ h$. Then necessarily $h(Y) \subset E$, but since $Y$ is mixing, so is $h(Y)$, and we actually have $h(Y) = \{\INF 0 \INF\}$. Now the unique morphism $u : Y \to \{ \INF 0 \INF \}$ with $h = i \circ u$ is simply the codomain restriction of $h$, and we are done.

Finally, let $X = \B^{-1}((10^*20^*)^*)$ and $Y = \{0,1\}^\Z$, which are transitive (even mixing) sofic shifts, and define $f, g : X \to Y$ as follows: $f$ is the symbol map $0, 1 \mapsto 0$ and $2 \mapsto 1$, while $g$ is again the all-$0$ map. Now we have $E = \B^{-1}(0^*10^*)$, and the inclusion of $\{\INF 0 \INF\}$ is the equalizer of $f$ and $g$ in T3, as above, but it is not isomorphic to their equalizer in K3.
\end{example}

The following collection of results characterizes the equalizers of the transitive categories. They are referred to in Proposition~\ref{prop:RegMonics} of Section~\ref{sec:Morphisms}.

\begin{proposition}
\label{prop:Equalizers}
Let $f, g : X \to Y$ be parallel morphisms, and denote $E = \{ x \in X \;|\; f(x) = g(x) \}$.
\begin{itemize}
\item In (M/P)2, $f$ and $g$ have an equalizer if and only if $E$ has at most one mixing component $E'$, and then it is the inclusion of $E'$ into $X$, or the empty map $\epsilon : \emptyset \to X$.
\item In (M/P)3, $f$ and $g$ have an equalizer if and only if $E$ has at most one maximal mixing sofic subshift $E'$, and then it is the inclusion of $E'$ into $X$, or the empty map $\epsilon : \emptyset \to X$.
\item In T2, $f$ and $g$ have an equalizer if and only if $E$ is transitive, and then it is the inclusion of $E$ into $X$.
\item In T3, $f$ and $g$ have an equalizer if and only if $E$ has a single transitive component $E'$, and then it is the inclusion of $E'$ into $X$.
\end{itemize}
\end{proposition}

Recall that maximal mixing subshifts and mixing components of sofic shifts are different notions in general.

\begin{proof}
We prove the claim in the case of M2. The others are simply variations of the same idea, except in the case of T2 we also need the fact that an SFT with only one transitive component is transitive.

Suppose first that such an $E'$ exists, denote by $i : E' \to X$ the inclusion map, and let $h : Z \to X$ be any morphism with $f \circ h = g \circ h$. This implies $h(Z) \subset E$. Now, $h(Z)$ is a mixing sofic subshift of $E$, and is thus contained in one of its transitive components, which must be mixing by Lemma~\ref{lem:MixingInComponent}, and hence equals $E'$. Then $h$ factors uniquely through $i$. Also, if no transitive component of $E$ is mixing, then we must have $h = \epsilon$.

Conversely, suppose that $E$ has two mixing components, and let $h : Z \to X$ be a morphism with $f \circ h = g \circ h$. We show that $h$ is not an equalizer of $f$ and $g$. Namely, $h(Z)$ is a mixing subshift of $E$, it is contained in some mixing component of $E$. Then the inclusion map of any other mixing component does not factor through $h$, and we are done.
\end{proof}

We then move to colimits. The categories (K/T/M/P)(2/3) and K4 do have initial objects, as it is easy to see that in (K/T/M)(2/3) and K4, they are the empty subshifts, and in P(2/3), they are the trivial subshifts $\{\INF 0 \INF\}$. The trivial subshifts are thus zero objects in P(2/3), that is, both initial and terminal. Binary coproducts also exist in K(2/3/4): the coproduct of two objects $X$ and $Y$ is the disjoint union $X \mathop{\dot{\cup}} Y$ together with the inclusion maps $i_1 : X \to X \mathop{\dot{\cup}} Y$ and $i_2 : Y \to X \mathop{\dot{\cup}} Y$.

The case of coequalizers is more subtle, as they exist for some parallel morphism pairs, but not for others, and even relatively simple cases in the mixing categories require a significant analysis. As with split epimorphisms, we present coequalizers in a separate subsection.

\subsection{Coequalizers}

As an introduction to the notion of coequalizers, we discuss the categorical notions of kernels and cokernels. These notions only make sense in categories with zero objects, or in our case, P(2/3). Recall that a morphism that factors through a zero object is called a zero morphism, and in P(2/3), they are exactly the trivial block maps $f : X \to Y$ with $f(X) = \{p(Y)\}$.

\begin{definition}
Let $\cat$ be a category with a zero object. The \emph{kernel} of a morphism $f : X \to Y$ is the equalizer of $f$ and the zero morphism $0_{X Y} : X \to Y$, and its \emph{cokernel} is their coequalizer, if these exist.
\end{definition}

Note that our definition for the kernel set of a function in Section~\ref{sec:Defs} is analogous to the notion of a kernel pair, while the categorical kernel is the generalization of the kernel of a group homomorphism or a linear function. Since the categories P(2/3) have zero objects by our earlier discussion, we can talk about kernels and cokernels. By Proposition~\ref{prop:Equalizers}, the kernel of a morphism $f : X \to Y$ in P(2/3) is the inclusion map of the maximal mixing sofic subshift of $f^{-1}(p(Y))$ into $X$, if one exists. Cokernels, on the other hand, exist only in trivial cases.

\begin{proposition}
In the pointed categories P(2/3), a morphism $f : X \to Y$ has a cokernel if and only if it is either surjective or $0_{X Y}$.
\end{proposition}

\begin{proof}
Let $Z$ be a zero object. We first note that the cokernel of the zero map $0_{X Y}$ is trivially $\ID_Y$, since for every morphism $g : Y \to Q$ with $g \circ 0_{X Y} = g \circ 0_{X Y}$ (that is, for any morphism whatsoever), there exists a unique morphism $u : Y \to Q$ ($g$ itself) with $g = u \circ \ID_Y$. Also, the cokernel of a surjection $f : X \to Y$ is the zero map $0_{Y Z} : Y \to Z$, since if a morphism $g : Y \to Q$ satisfies $g \circ f = g \circ 0_{X Y}$, then $g = 0_{Y Q}$, and in this case there trivially exists a unique map $u : Z \to Q$ (the zero map $0_{Z Q}$) such that $g = u \circ 0_{Y Z}$.

Suppose then that $f$ is neither surjective nor trivial, and denote $p(Y) = \INF a \INF$. Since $Y$ and $f(X)$ are mixing sofic shifts, there then exists a point $\INF w \INF \in Y \setminus f(X)$ and another point $\INF a v a \INF \in f(X)$ with $v, w \notin a^*$. Let $g : Y \to Q$ be a morphism with radius $r \in \N$ such that $g \circ f = 0_{X Q}$. We show that $g$ is not a cokernel of $f$. For this, define $y_1 = \INF a w a^r v a \INF$ and $y_2 = \INF a w a \INF$, so that necessarily $g(y_1) = g(y_2)$. Denote $P = \{0,1\}^\Z$, and define the block map $h : Y \to P$ with radius $r' = r + |v| + |w|$ by
\[ h(y)_0 =
\left\{
	\begin{array}{ll}
		0, & \mbox{if~} y_{[0,r'-1]} \in \B_{r'}(f(X)), \\
		1, & \mbox{otherwise}
	\end{array}
\right. \]
for all $y \in Y$. Then $h(y_1) \neq h(y_2)$, but $h \circ f = 0_{X P}$. Now there is no block map $u : Q \to P$ such that $h = u \circ g$, and thus $g$ is not a cokernel of $f$.
\end{proof}

Now we move on to the most general categories K(2/3/4). We begin by mentioning that there is an abstract characterization of the existence and nature of general coequalizers, although we prefer more hands-on techniques for computing them in the rest of this section, since one of the morphisms will always be an identity map.

\begin{proposition}
Let $f, g : X \to Y$ be a parallel pair of morphisms in K(3/4). Let $R \subset Y^2$ be the intersection of all local subSFT equivalence relations containing $(f(x), g(x))$ for all $x \in X$. Then $f$ and $g$ have a coequalizer if and only if $R$ is a subSFT of $Y^2$.
\end{proposition}

\begin{proof}
Suppose first $R$ is a subSFT. Then, there exists a finite family $(R_i)_{i=0}^{k-1}$ of local subSFT equivalence relations of $Y$ such that $R = \bigcap_{i=0}^{k-1} R_i$, and then $R$ is also local. Then there exists a surjective block map $h : Y \to Z$ with $\Ker{h} = R$, and clearly $h \circ f = h \circ g$. Also, $Z$ is an object of K(3/4). Suppose that $t : Y \to T$ is such that $t \circ f = t \circ g$. Since $\Ker{t} \subset Y^2$ is a local subSFT equivalence relation containing $(f(x), g(x))$ for all $x \in X$, it also contains $R = \Ker{h}$. Corollary~\ref{cor:ContConnector} then gives the unique morphism $u : Z \to T$ with $t = u \circ h$. Thus $h$ is the coequalizer of $f$ and $g$.

Suppose then that $R$ is not a subSFT, and let $h : Y \to Z$ be any morphism with $h \circ f = h \circ g$, so that $R \subsetneq \Ker{h}$. Then there is a local subSFT equivalence relation $R' \subset Y^2$ with $R \subset R' \subsetneq \Ker{h}$, and there exists a morphism $t : Y \to T$ with $\Ker{t} = R'$. Then $t \circ f = t \circ g$, but there exists no morphism $u : Z \to T$ such that $t = u \circ h$. Thus $h$ is not the coequalizer of $f$ and $g$, and since $h$ was arbitrary, the coequalizer does not exist.
\end{proof}

Next, we study the coequalizers of pairs $(\ID_X,f)$, where $f : X \to X$ is an endomorphism of the object $X$. This is an especially interesting case from the dynamical systems perspective, since the morphisms $g : X \to Y$ such that $g \circ f = g$ are exactly those that identify the orbits of $f$. Intuitively, such a $g$ should be viewed as a `conserved local property', and the coequalizer of $f$ and $\ID_X$, if it exists, is then a `universal' such property, that is, one that subsumes all others. The main result of this subsection is Theorem~\ref{thm:Coequalizers}, which basically states that coequalizers of $(\ID_X, f)$-pairs in K3 are uncomputable. Our main results below will be about endomorphisms $f : X \to X$ where $X$ is a full shift, and we thus use the convention from cellular automata literature of calling a bijective cellular automaton \emph{reversible}.

\begin{example}
Let $X \subset S^\Z$ be a mixing object of K(2/3), and let $f : X \to X$ be an endomorphism with a \emph{spreading state}, that is, a state $s \in S$ such that $s \in \{x_0, x_1\}$ implies $x_0 = s$. Then, every block map $g : X \to Y$ with $g \circ f = g$ must satisfy $g(x) = g(\INF s \INF)$ for all $x \in X$: If $r \in \N$ is a radius for $g$, let $x' \in X$ be such that $x'_{[-r,r]} = x_{[-r,r]}$, but $x'_n = s$ for some $n > r$. Then we have
\[ g(x)_0 = g(x')_0 = g(f^{n+r}(x'))_0 = g(\INF s \INF)_0. \]
It is then easy to see that $f$ and $\ID_X$ have a coequalizer in the respective category, namely the zero morphism $0_{XZ} : X \to Z$, where $Z$ is the trivial subshift. The same result holds if $f$ is \emph{nilpotent}, that is, satisfies $f^n(X) = \{\INF s \INF\}$ for some $n \in \N$ and $s \in S$.
\end{example}

A self-map $f$ of a set $X$ is called eventually periodic, if there exist $k \in \N$ and $p > 0$ such that $f^k = f^{k+p}$. Then $p$ is called an eventual period of $f$. \emph{The} eventual period of $x \in X$ is the smallest positive $p$ with $f^k(x) = f^{k+p}(x)$ for some $k \in \N$. As another example of coequalizers, we characterize those eventually periodic morphisms $f : X \to X$ for which the coequalizer of $f$ and $\ID_X$ exists, where $X$ is any object of the mixing categories. For this, we define some topological tools.

\begin{definition}
Let $(X, d)$ be a compact metric space, and denote by $2^X$ the set of all closed subsets of $X$. The \emph{Hausdorff metric} is defined by
\[ d_H(A, B) = \max \{ \sup_{a \in A} d(a,B), \sup_{b \in B} d(b,A) \} \]
for all $A, B \in 2^X$, where $d(a,B) = \inf_{b \in B} d(a,b)$ for $a \in X$ and $B \in 2^X$.
\end{definition}

It is well known to see that $d_H$ indeed is a metric on $2^X$, and $(2^X, d_H)$ is also compact. In the context of subshifts, we define
\[ A_{[i,i+r-1]} = \{ a_{[i,i+r-1]} \;|\; a \in A \} \subset S^r \]
for all $A \subset S^\Z$, $i \in \Z$ and $r \in \N$. Then for $n \in \N$, two closed sets $A, B \subset S^\Z$ are $2^{-n}$-close with respect to the Hausdorff metric if and only if $A_{[-n,n]} = B_{[-n,n]}$.

\begin{proposition}
\label{prop:EventuallyPeriodic}
Let $X$ be a mixing SFT object of K(3/4), and let $f : X \to X$ be eventually periodic. Then $f$ and $\ID_X$ have a coequalizer if and only if every $x \in X$ has the same eventual period.
\end{proposition}

If $f$ has this property, we say it is \emph{visibly eventually periodic}. The proposition could be proved combinatorially, but we present a topological proof using the Hausdorff metric.

\begin{proof}
Suppose first that $f$ is visibly eventually periodic with the eventual period $p > 0$, and consider the set
\[ \tilde X = \{ \{x, f(x), \ldots, f^{p-1}(x)\} \;|\; x \in X, f^p(x) = x \} \subset 2^X, \]
which we metrize with the Hausdorff metric. It is also equipped with the natural shift action $\tilde \sigma : \tilde X \to \tilde X$. Let $k \in \N$ be such that $f^k(x) = f^{k+p}(x)$ for all $x \in X$. We define the function $g : X \to \tilde X$ by $g(x) = \{ f^k(x), \ldots, f^{k+p-1}(x) \}$. First, $\tilde X$ is zero-dimensional, since it has the clopen base consisting of the sets $\{ A \in \tilde X \;|\; A_{[-n,n]} = W \}$ for all $n \in \N$ and $W \subset S^{2n+1}$. By the properties of the Hausdorff metric, $\tilde X$ is a compact metric space.

We proceed to show that $\tilde \sigma$ is expansive, so that $(\tilde X, \tilde \sigma)$ is a subshift. Assume the contrary, and suppose that for all $n \in \N$, there exist two sets $A^n = \{ a^n, \ldots, f^{p-1}(a^n) \} \in \tilde X$ and $B^n = \{ b^n, \ldots, f^{p-1}(b^n) \} \in \tilde X$ such that $A^n \neq B^n$, but $d_H(\tilde \sigma^m(A^n), \tilde \sigma^m(B^n)) < 2^{-n}$ for all $m \in \Z$. We can suppose that $a^n_{[-n,n]} = b^n_{[-n,n]}$, but since $A^n \neq B^n$, we have $a^n \neq b^n$, and thus $\sigma^m(a^n)_{[-n,n]} \neq \sigma^m(b^n)_{[-n,n]}$ for some $m \in \Z$, which we may assume to be positive. Let $m_n \geq 0$ be minimal such that $\sigma^{m_n}(a^n)_{[-n,n]} = \sigma^{m_n}(b^n)_{[-n,n]}$ but $\sigma^{m_n+1}(a^n)_{[-n,n]} \neq \sigma^{m_n+1}(b^n)_{[-n,n]}$, so that there is $q_n \in \{1, \ldots, p-1\}$ such that $\sigma^{m_n+1}(a^n)_{[-n,n]} = \sigma^{m_n+1}(f^{q_n}(b^n))_{[-n,n]}$. Now, the sequence $(a^n, b^n, f^{q_n}(b^n))_{n \in \N}$ has a limit point $(a, b, f^q(b))$, where $q \in \{1, \ldots, p-1\}$. Since $b^n_{[-n-1,n]} = f^q(b^n)_{[-n,n-1]}$ for all $n \in \N$, we have $b = f^q(b)$, contradicting the minimality of $p$. Thus $g$ is a surjective block map to the sofic shift $\tilde X$.

Suppose now that $h : X \to Y$ is any morphism with $h \circ f = h$. This clearly implies $\Ker{g} \subset \Ker{h}$, so by Corollary~\ref{cor:ContConnector} we have a unique morphism $u : \tilde X \to Y$ with $h = u \circ g$. Thus $g$ is a coequalizer of $f$ and $\ID_X$.

Suppose then that $f$ is not visibly periodic, so that there exists $x \in X$ with eventual period $q$ properly dividing $p$, and let $g : X \to Y$ be such that $g \circ f = g$. Let $r \in \N$ be larger than the radius of $g$ and the window size of $X$, and let $w,v \in \B(X)$ be such that $\INF w \INF \in X$ has least $f$-period $p$, $\INF v \INF \in X$ has least $f$-period $q$, and $x' = \INF w v^r w \INF \in X$. Such words exist because of the mixingness of $X$. Define $\INF u \INF = f^q(\INF w \INF)$, so that $x'' = \INF w v^r u \INF \in X$. Now we clearly have $g(x') = g(x'')$, since $x''$ appears locally as either $x'$ or $f^q(x')$, which $g$ does identify. Define the block map $h : X \to \{0,1\}^\Z$ by $h(x)_0 = 1$ if and only if $f^k(x)_{[0,2|w|+r|v|-1]} = w v^r u$ for some $k \in \{0, \ldots, p-1\}$. We clearly have $h \circ f = f$, but $h(x') \neq h(x'')$. Thus $h$ does not factor through $g$, which implies that $g$ is not the coequalizer of $f$ and $\ID_X$.
\end{proof}

For the next set of results, we need some dynamical notions.

\begin{definition}
A set of words $W \subset \B_\ell(X)$ is \emph{visibly blocking} for $f$ if
\begin{itemize}
\item for all $x \in X$, if $x_{[0,\ell-1]} \in W$, then $f(x)_{[0,\ell-1]} \in W$, and
\item for all $x, y \in X$ such that $x_{[0,\ell-1]} \in W$ and $x_i = y_i$ for all $i \geq 0$ ($i \leq \ell-1$), we have $f^n(x)_i = f^n(y)_i$ for all $n \in \N$ and $i \geq \ell$ ($i < 0$, respectively).
\end{itemize}
To a visibly blocking set $W \subset \B_\ell(X)$, we attach its \emph{characteristic function} $\chi_W : X \to \{0,1\}^\Z$, the block map defined by $\chi_W(x)_0 = 1$ if and only if $x_{[0, \ell-1]} \in W$.
\end{definition}

The following lemma is useful in general.

\begin{lemma}
\label{lem:OrderMap}
Suppose $X$ is a mixing sofic shift, $f : X \to X$ is reversible and $g : X \to P^\Z$, where $P$ is a partially ordered set, and $g(x)_0 \geq g(f(x))_0$ for all $x \in X$. Then $g \circ f = g$.
\end{lemma}

\begin{proof}
Suppose that $g(x)_0 \neq g(f(x))_0$ for some $x \in X$, so that we actually have $g(x)_0 > g(f(x))_0$. Since the set of spatially periodic points is dense in $X$, we may assume $x$ to be such a point. But then $x$ is also $f$-periodic with some period $p \in \N$, so that
\[ g(x)_0 > g(f(x))_0 \geq \ldots \geq g(f^p(x))_0 = g(x)_0, \]
a contradiction.
\end{proof}

Since $\chi_W$ clearly satisfies the above condition, we have the following.

\begin{corollary}
If $X$ is a mixing sofic shift, $f : X \to X$ is reversible and $W$ is a visibly blocking set for $f$, then $\chi_W \circ f = \chi_W$.
\end{corollary}

\begin{lemma}[Theorem 4.5 of \cite{Lu10a}]
The classes of mixing and nonsensitive reversible cellular automata on full shifts are recursively inseparable.
\end{lemma}

In general, nonsensitivity of a cellular automaton on a one-dimensional full shift is equivalent to the existence of \emph{blocking words} \cite{BlTi00}, of which elements of visibly blocking sets are a special case. However, in the proof of this particular theorem, if the automaton is nonsensitive, there always exists a visibly blocking set: The blocking words in the construction are bordered areas on which valid periodic runs of a Turing machine are simulated. No information can enter or escape such areas and their borders never move.

The actual result we will use is thus the following.

\begin{lemma}[Proved as Theorem 4.5 of \cite{Lu10a}]
\label{lem:NiceBlocks}
The classes of reversible mixing cellular automata and reversible cellular automata with visibly blocking sets on full shifts are recursively inseparable.
\end{lemma}

\begin{lemma}
\label{lem:NotPeriodic}
Let $X = S^\Z$. If the reversible CA $f : X \to X$ has a visibly blocking set but is not periodic, then $f$ and $\ID_X$ have no coequalizer in K3.
\end{lemma}

\begin{proof}
Let $W \subset S^\ell$ be a visibly blocking set, and let $w \in W$. Since $\INF w \INF$ is spatially periodic, it is temporally periodic with some least period $p \in \N$. Let $v \in S^*$ be such that no $x \in X$ with $x_{[0,|v|-1]} = v$ satisfies $f^p(x) = x$. Then $\INF(wvw)\INF$ is a spatially periodic point, and thus also has a least temporal period $q \in \N$. We easily see that $p \neq q$ and $p | q$. Denote $f^p(\INF(wvw)\INF) = \INF u \INF$, where $|u| = |wvw|$.

We proceed as in the proof of Proposition~\ref{prop:EventuallyPeriodic}. So, suppose that $g : X \to Y$ with some radius $r \in \N$ is such that $g \circ f = g$, and consider the points $x' = \INF (wvw) . w^r(wvw) \INF$ and $x'' = \INF u . w^r (wvw) \INF$. By the definition of $g$, we then have $g(x') = g(x'')$, since $x''$ locally appears as either $x'$ or $f^p(x')$, which $g$ cannot distinguish.

We then construct a block map $h : X \to Z$ with $h \circ f = h$ that does not factor through $g$, proving that $g$ is not the coequalizer of $f$ and $\ID_X$. For this, denote $m = (r+2)|w|+2|v|$, and let $P = 2^{S^m}$ and $Z = P^\Z$. The block map $h$ is defined as follows for all $x \in X$. If we have $x_{[-\ell,-1]}, x_{[m, m+\ell-1]} \in W$, then define
\[ h(x)_0 = \{f^k(x)_{[0,m-1]} \;|\; k \in \N\}. \]
Since $W$ is a visibly blocking set, this finite set of words depends only on $x_{[-\ell, m+\ell-1]}$. Otherwise, define $h(x)_0 = \emptyset$. It is now clear that $h(x') \neq h(x'')$, so that $h$ does not factor through $g$. Next, note that for $x \in X$, the condition $h(x)_0 = \emptyset$ depends only on $\chi_W(x)$, and since $\chi_W \circ f = \chi_W$, we have $h(x)_0 = \emptyset$ if and only if $h(f(x))_0 = \emptyset$. Then it is easy to see that $h(f(x))_0 \subset h(x)_0$ for all $x \in X$, so by Lemma~\ref{lem:OrderMap} we have $h \circ f = h$.
\end{proof}

\begin{proposition}
\label{prop:ChainTrans}
Let $X$ be a mixing sofic shift and $T = \{\INF 0 \INF\}$. For a reversible cellular automaton $f : X \to X$, the map $0_{XT}$ is a coequalizer of $f$ and $\ID_X$ in K3 if and only if $f$ is chain transitive.
\end{proposition}

\begin{proof}
Suppose first that $f$ is chain transitive, and let $g : X \to Y$ be such that $g \circ f = g$. Let $r \in \N$ be its radius, and let $u,v \in \B_{2r+1}(X)$. By chain transitivity, there exists a chain $x^1, \ldots, x^k \in X$ such that $x^1_{[-r,r]} = u$, $x^k_{[-r,r]} = v$ and for all $i$, $f(x)^i_{[-r,r]} = x^{i+1}_{[-r,r]}$. This implies $g(x^1)_0 = g(x^2)_0 = \cdots = g(x^k)_0$, and since $u$ and $v$ were arbitrary, we have $g(x)_0 = g(x')_0$ for all $x, x' \in X$. Then there is a unique morphism $u : T \to Y$ (the symbol map $0 \mapsto g(x)_0$) with $g = u \circ h$.

Suppose then that $f$ is not chain transitive, so that there exist $n \in \N$ and $u,v \in \B_n(X)$ such that no chain from $u$ to $v$ exists in $X$. Denote $P = 2^{\B_n(X)}$, and define the block map $h : X \to P^\Z$ by
\[ h(x)_0 = \{ v \in \B_n(X) \;|\; \exists \mbox{~chain from $x_{[0,n-1]}$ to $v$} \} \]
for all $x \in X$. Since $h(f(x))_0 \subset h(x)_0$ for all $x \in X$, we have $h \circ f = h$ by Lemma~\ref{lem:OrderMap}. Since there is no chain from $u$ to $v$, the map $h$ is not trivial, and thus $0_{XT}$ is not the coequalizer of $f$ and $\ID_X$.
\end{proof}

\begin{theorem}
\label{thm:Coequalizers}
Let $T = \{\INF 0 \INF\}$. The classes of reversible cellular automata $f$ on full shifts $X = S^\Z$ for which $0_{XT}$ is a coequalizer of $f$ and $\ID_X$ in K3, and of those for which no coequalizer exists, are recursively inseparable.
\end{theorem}

\begin{proof}
Suppose on the contrary that there exists a Turing machine $M$ that accepts automata of the first kind, and rejects those of the second. We use $M$ to recursively separate the classes of mixing reversible CA, and those that have a visibly blocking set, contradicting Lemma~\ref{lem:NiceBlocks}.

Let $X = S^\Z$, and let $f : X \to X$ be a reversible cellular automaton with radius $r \in \N$. First, if $f$ is periodic with period $p \leq |S|$ (which is easy to decide), we answer `visibly blocking set'. If $f$ does not have a low period, we give $f$ as input to $M$, and return `mixing' if $M$ answers `trivial coequalizer', and `visibly blocking set' if $M$ answers `no coequalizer'.

We prove the correctness of this algorithm. First, the algorithm always halts, since $M$ does. Second, suppose that $f$ is mixing. Then it is in particular chain transitive and not periodic, so that the trivial map is a coequalizer of $f$ and $\ID_X$ by Proposition~\ref{prop:ChainTrans}. Thus the above algorithm correctly returns `mixing'. Next, suppose $f$ has a visibly blocking set. If $f$ has a small period, this is noticed in the first part of the algorithm, and we correctly return `visibly blocking set'. Suppose thus that $f$ has no small period. If it has no period whatsoever, then by Lemma~\ref{lem:NotPeriodic}, $f$ and $\ID_X$ have no coequalizer, and the algorithm correctly returns `visibly blocking set'.

Finally, suppose $f$ has a least period $p > |S|$. Since every unary point has period at most $|S|$, $f$ is not visibly eventually periodic, and Proposition~\ref{prop:EventuallyPeriodic} states that $f$ and $\ID_X$ have no coequalizer. Thus the algorithm correctly returns `visibly blocking set', and we are done.
\end{proof}

This shows that the computation of colimits is impossible in general, which is in sharp contrast with Proposition~\ref{prop:LimitsExist}.

\begin{corollary}
Given a diagram in K3 and its co-cone, it is undecidable whether this co-cone is a colimit of the diagram.
\end{corollary}

In Theorem~\ref{thm:Coequalizers}, we proved in particular that it is undecidable whether the pair $(f, \ID_X)$ has the trivial morphism $0_{XT}$ as its coequalizer. One may ask whether $0_{XT}$ can be replaced with some other morphism, and in particular whether there is an analogue of Rice's theorem for coequalizers. We provide some evidence to the contrary with the following result.

\begin{proposition}
Let $X$ be an object of K(3/4), and $f : X \to X$ a morphism. Then $\ID_X$ is a coequalizer for $f$ and $\ID_X$ in K(3/4) if and only if $f = \ID_X$.
\end{proposition}

\begin{proof}
First, it is easy to see that $\ID_X$ is the coequalizer of the pair $(\ID_X, \ID_X)$. On the other hand, if $f \neq \ID_X$, there exists $x \in X$ with $f(x) \neq x$. Then for any map $g : X \to Y$ with $g \circ f = g$, we must have $g(x) = g(f(x))$, so $g$ is not injective, in particular $g \neq \ID_X$.
\end{proof}

\section{Properties of the Symbolic Categories}
\label{sec:Properties}

In Proposition~\ref{prop:LimitsExist}, we showed that the categories K(2/3/4) are finitely complete. The goal of this section is to extend these results as much as possible. The main results here are that K3 and K4 are also regular, coherent and extensive, while K2 is only extensive. We also show that none of these categories is exact.

We first consider the case of regularity, as it is a prerequisite of both coherency and exactness. Intuitively, a regular category is one where every morphism has a well-behaved image object, so it should morally be true that K3 and K4 are regular, but K2 is not. Namely, the regularity of a category is equivalent to the following two conditions: First, for all morphisms $f : X \to Y$, there exists an \emph{image factorization} $f = m \circ e$, where $e : X \to Z$ is a morphism and $m : Z \to Y$ a monomorphism, such that for all other such factorizations $f = m' \circ e'$ we have $m \leq m'$ as subobjects of $Y$. Second, the image factorizations are stable under pullback. We could prove the regularity of K3 using this condition, but follow the definition instead, since our argument then also proves Proposition~\ref{prop:RegularEpis}.

\begin{proposition}
\label{prop:Regularity}
The categories K3 and K4 are regular.
\end{proposition}

\begin{proof}
First, the categories are finitely complete by Proposition~\ref{prop:LimitsExist}. Let then $f : X \to Y$ be a morphism in K(3/4), and let $p_1, p_2 : K \to X$ be the projections from the kernel pair $K = \{ (x,x') \in X^2 \;|\; f(x) = f(x') \}$ of $f$. We show that the codomain restriction $f : X \to f(X)$ of $f$, which exists in K(3/4), is a coequalizer for $p_1$ and $p_2$. Let thus $g : X \to Z$ be such that $g \circ p_1 = g \circ p_2$, or equivalently, $\Ker{f} \subset \Ker{g}$. By Corollary~\ref{cor:ContConnector}, there exists a unique morphism $u : f(X) \to Z$ such that $g = u \circ f$, and thus $f$ really is the coequalizer of $p_1$ and $p_2$. This also shows that in K(2/3/4), every epimorphism is regular, being the coequalizer of its kernel pair, and thus proves Proposition~\ref{prop:RegularEpis}.

Let then $f : X \to Y$ be a regular epimorphism in K(3/4), that is, a surjective block map, and let $g : Z \to Y$ be arbitrary. We show that the pullback of $f$ along $g$, that is, the projection $p_2 : X \times_Y Z \to Z$, is also a regular epimorphism. For that, let $z \in Z$ be arbitrary. Since $f$ is surjective, there exists $x \in X$ with $f(x) = g(z)$, and then $(x, z) \in X \times_Y Z$, thus $p_2(x, z) = z \in p_2(X \times_Y Z)$. This shows that $p_2$ is surjective, and thus a regular epimorphism.
\end{proof}

This proof does not work in K2, since the codomain restriction cannot be performed for morphisms whose image is proper sofic.

\begin{example}
Let $f : X \to Y$ be a morphism of K2 such that $f(X)$ is proper sofic. For all $n \in \N$, let $Y_n \subset Y$ be the subshift defined by the forbidden words of $Y$, together with all the words of length at most $n$ that do not occur in $f(X)$. Then each $Y_n$ is an SFT with $Y_{n+1} \subset Y_n$ and $f(X) = \bigcap_{n \in \N} Y_n$. Then for any factorization $f = m \circ e$, where $m : Z \to Y$ is a monomorphism, we have $Y_n \subset m(Z)$ for some $n \in \N$ since $m(Z)$ is an SFT. Then the factorization $f = m_n \circ e_n$, where $e_n : X \to Y_n$ is the codomain restriction of $f$ and $m_n : Y_n \to Y$ the inclusion map, does not satisfy $m_n \leq m$, and thus no image factorization for $f$ exists. This shows that K2 is not a regular category.
\end{example}

Next, we turn to coherency. By definition, since K2 is not regular, it cannot be coherent either. The intuition for coherency is the existence of well-behaved binary unions of subobjects, so K3 and K4 should have this property.

\begin{proposition}
\label{prop:Coherency}
The categories K3 and K4 are coherent.
\end{proposition}

\begin{proof}
Let $f : Y \to X$ and $g : Z \to X$ be monomorphisms in K(3/4), that is, injective block maps. Without loss of generality we assume $Y$ and $Z$ to be subshifts of $X$, and $f$ and $g$ to be the respective inclusion maps.

We claim that the inclusion map $i : Y \cup Z \hookrightarrow X$ is a least upper bound for $f$ and $g$ in $\Sub(X)$. First, $i$ is an upper bound, since the inclusion maps $i_1 : Y \hookrightarrow Y \cup Z$ and $i_2 : Z \hookrightarrow Y \cup Z$ satisfy $f = i \circ i_1$ and $g = i \circ i_2$. Suppose then that an inclusion map $h : Q \hookrightarrow X$ is another upper bound for $f$ and $g$, so that there exist morphisms $h_1 : Y \to Q$ and $h_2 : Z \to Q$ such that $f = h \circ h_1$ and $g = h \circ h_2$. Since $h_1$ and $h_2$ must be monic, this means just that $Y \subset Q$ and $Z \subset Q$, so that $Y \cup Z \subset Q$, and then $h$ factors through $i$.

Let then $k : T \to X$ be any morphism, and recall the definition of the base change functor $k^* : \Sub(X) \to \Sub(T)$. Now, the pullback $k^*(f) : Y \times_X T \to T$ is the second projection from the fiber product $\{ (y,t) \in Y \times T \;|\; k(t) = y \}$, and thus isomorphic to the inclusion of $k^{-1}(Y)$ into $T$. Then, the union of the pullbacks of $f$ and $g$ is isomorphic to the inclusion of $k^{-1}(Y) \cup k^{-1}(Z)$ into $T$, and the pullback of their union to the inclusion of $k^{-1}(Y \cup Z)$ into $T$, and since preimages commute with unions, these are the same subobject.
\end{proof}

We also note that in finitely complete categories (which K(2/3/4) are), every pair of subobjects $f : Y \to X$ and $g : Z \to X$ of $X$ has a greatest lower bound, given by their pullback, which is preserved by the base change maps. In our case, this corresponds to the intersection of subshifts, and indeed the classes of SFTs, sofic shifts and all subshift are closed under intersection.

Next, we study the extensiveness of K(2/3/4). This property intuitively corresponds to the existence of all disjoint unions as well-behaved coproducts. Since all three categories have coproducts that are set-theoretically disjoint unions, we should expect them to be extensive, and again this is indeed the case.

\begin{proposition}
The categories K(2/3/4) are extensive.
\end{proposition}

\begin{proof}
First, we saw in Section~\ref{sec:Limits} that the categories have all finite coproducts, given by the symbol-disjoint unions of subshifts. It remains to consider the commutative diagram
\begin{center}
\begin{tikzpicture}[node distance=1.5cm, auto]
	\node (a) {$A$};
	\node (ab) [right of=a] {$A \mathop{\dot{\cup}} B$};
	\node (b) [right of=ab] {$B$};
	\node (x) [above of=a]{$X$};
	\node (z) [right of=x] {$Z$};
	\node (y) [right of=z] {$Y$};
	\draw[->] (a) to node {$i_1$} (ab);
	\draw[->] (b) to node [swap] {$i_2$} (ab);
	\draw[->] (x) to node [swap] {$j_1$} (z);
	\draw[->] (y) to node {$j_2$} (z);
	\draw[->] (x) to node {$f$} (a);
	\draw[->] (z) to node {$h$} (ab);
	\draw[->] (y) to node {$g$} (b);
\end{tikzpicture}
\end{center}
where $i_1$ and $i_2$ are the coproduct inclusion maps of $A$ and $B$, respectively. First, suppose that the two squares are pullback diagrams. Then $j_1$ and $j_2$ are monomorphisms, and by the argument in the final paragraph of the proof of Proposition~\ref{prop:Coherency} (which is valid also in K2), we have $j_1(X) = h^{-1}(A)$ and $j_2(Y) = h^{-1}(B)$. Then $Z = h^{-1}(A \mathop{\dot{\cup}} B) = h^{-1}(A) \mathop{\dot{\cup}} h^{-1}(B) = j_1(X) \mathop{\dot{\cup}} j_2(Y)$. But this means exactly that $Z$ is a coproduct of $X$ and $Y$ with the injections $j_1$ and $j_2$.

Suppose conversely that $Z$ is a coproduct of $X$ and $Y$ with the injections $j_1$ and $j_2$. For all $x \in X$, define $\phi(x) = (f(x),j_1(x)) \in A \times Z$. We then have $i_1(a) = h(z)$ for $(a,z) \in A \times Z$ if and only if $z \in j_1(X)$ and $(a,z) = \phi(j_i^{-1}(z))$, so that $\phi(X)$ (which is an object also in K2, since $j_1$ is injective) with the corresponding projections is the pullback of $i_1$ and $h$. Since $j_1$ is injective, so is $\phi$, and then it is clear that $X$, $f$ and $j_1$ form a pullback for $i_1$ and $h$.
\end{proof}

Finally, we consider exactness, which intuitively corresponds to the property that every equivalence relation has a well-defined quotient. The categorical formalization of equivalence relations are congruences, and having a quotient object corresponds to being effective, that is, realized as a kernel pair of a morphism, which acts as the canonical projection map. In general, all equalizers in a regular category give rise to congruences, and exactness captures the converse situation.

In the categories K(2/3/4), all monomorphisms are injections, so congruences are essentially subshift equivalence relations (recall their definition from Section~\ref{sec:Prelims}). Since a subshift equivalence relation is effective if and only if it is a local subSFT equivalence relation, the category K3 is exact if and only if every sofic equivalence relation is a local subSFT equivalence relation. But this is blatantly false, as shown by the proper sofic relation on $\{0,1,2\}^\Z$ that equates two configurations $x$ and $y$ if and only if $x = y$ or $x = \tau(y)$, where $\tau : \{0,1,2\}^\Z \to \{0,1,2\}^\Z$ is the symbol permutation $(1 \; 2)$. Thus K3 is not an exact category, and neither is K4. Furthermore, Example~\ref{ex:NotLocal} shows that not all congruences are effective even in K2.

\section{Conclusions}

In this article, we have presented a study of the basic properties of thirteen natural symbolic categories. In many cases, we characterized basic categorical properties of morphisms. We also considered natural symbolic dynamical properties of objects, and showed that many of them can be expressed in the language of category theory. We then studied the limits and colimits of the categories, all of which except coequalizers correspond to very simple and natural constructions. Finally, we established some nice regularity properties for the categories K(2/3/4).

We consider the two main results of this paper to be the decidability of split epicness in K(2/3), and the uncomputability of coequalizers in the same categories. The first can be seen as a dual to the Extension Lemma, and its proof uses a Ramsey-type argument together with the classical Marker Lemma. The main idea is that given any existing section, one can use the Marker Lemma to construct another section with a bounded radius, considering periodic and nonperiodic parts separately, and in this sense the construction is similar to the proofs of Factor Theorem and Embedding Theorem in \cite{LiMa95} and the main result of \cite{Sa12b}. The second result is basically an application of the main construction in \cite{Lu10a}, together with some dynamical characterizations and conditions on the existence of coequalizers.

However, many natural problems remain unsolved, including decidability of split monicness in (T/K)(2/3), a more natural characterization of monicness in M(2/3), categoricity of many symbolic dynamical properties of morphisms, and existence of coequalizers of some restricted classes of morphism pairs. The study of these problems, and category theoretical notions in general, would probably give rise to many more interesting problems.

It would also be interesting to consider other symbolic categories, such as the category of coded systems or that of minimal subshifts. The category of minimal subshifts is probably very different from all of the categories considered here. For example, every block map between minimal subshifts is automatically surjective, so every morphism of the category is epic. The multidimensional setting is also a possible generalization. The two-dimensional analogue of K2 already contains subshifts without any periodic points, most properties of its objects are undecidable, and there are multiple nonequivalent generalizations of the mixing property of one-dimensional SFTs. Thus we expect the category of two-dimensional SFTs to behave much worse than K2 at least in terms of the categorical properties in Section~\ref{sec:Other}.

\section*{Acknowledgements}

We are thankful to Pierre Guillon for many useful discussions especially on monomorphisms. The first part of Example~\ref{ex:XORMonicness} is due to Silvio Capobianco, and was communicated to us by Guillon, although our proof is different. We are also thankful to Guillaume Theyssier for suggesting the addition of the category K4 of all subshifts. Finally, we would like to thank the anonymous referees for their valuable comments.

\section*{References}

\bibliographystyle{elsarticle-num}
\bibliography{../../../bib/bib}

\end{document}